\documentclass[11pt]{article}

\usepackage[latin1]{inputenc}
\usepackage{t1enc}
\usepackage{amsmath, amssymb, amsfonts, amsthm,bm, amsopn,epsfig}
\usepackage[none]{hyphenat}
\usepackage{tikz}
\usetikzlibrary{math,calc,intersections}
\usetikzlibrary{positioning,arrows,shapes,decorations.markings,decorations.pathreplacing,matrix,patterns}
\tikzstyle{vertex}=[circle,draw=black,fill=black,inner sep=0,minimum size=5pt,text=white,font=\footnotesize]

\usepackage{float}
\usepackage{graphicx}
\usepackage{caption}
\usepackage[margin=1in]{geometry}
\usepackage[toc,page]{appendix}
\usepackage{epstopdf}
\usepackage{etoolbox}
\usepackage[colorlinks=true]{hyperref}
\usepackage{dsfont}
\hypersetup{
	colorlinks=true,
	linkcolor=blue,
	filecolor=magenta,      
	urlcolor=cyan,
	citecolor=blue
}
\usepackage{cleveref}

\theoremstyle{plain}
\newtheorem{theorem}{Theorem}[section]

\newtheorem{claim}[theorem]{Claim}
\newtheorem{lemma}[theorem]{Lemma}

\theoremstyle{definition}

\newcommand{\precdot}{\prec\mathrel{\mkern-5mu}\mathrel{\cdot}}

\title{Dedekind's problem in the hypergrid}
\author{Victor Falgas-Ravry, Eero R\"aty, Istv\'an Tomon\thanks{Ume\r{a} University, \emph{e-mail}: \textbf{\{victor.falgas-ravry,eero.raty,istvan.tomon\}@umu.se}}}
\date{}

\begin{document}
	\sloppy 
	
	\maketitle

\begin{abstract}
Consider the partially ordered set on $[t]^n:=\{0,\dots,t-1\}^n$ equipped with the natural coordinate-wise ordering. Let $A(t,n)$ denote the number of antichains of this poset. The quantity $A(t,n)$ has a number of combinatorial interpretations: it is precisely the number of $(n-1)$-dimensional partitions with entries from $\{0,\dots,t\}$, and by a result of Moshkovitz and Shapira, $A(t,n)+1$ is equal to the $n$-color Ramsey number of monotone paths of length $t$ in 3-uniform hypergraphs. This has led to significant interest in the growth rate of $A(t,n)$.

A number of results in the literature show that $\log_2 A(t,n)=(1+o(1))\cdot \alpha(t,n)$, where $\alpha(t,n)$ is the width of $[t]^n$, and the $o(1)$ term goes to $0$ for $t$ fixed and $n$ tending to infinity. In the present paper, we prove the first bound that is close to optimal in the case where $t$ is arbitrarily large compared to $n$, as well as improve all previous results for sufficiently large $n$. In particular, we prove that there is an absolute constant $c$ such that for every $t,n\geq 2$,
$$\log_2 A(t,n)\leq \left(1+c\cdot \frac{(\log n)^3}{n}\right)\cdot \alpha(t,n).$$
This resolves a conjecture of Moshkovitz and Shapira. A key ingredient in our proof is the construction of a normalized matching flow on the cover graph of the poset $[t]^n$ in which the distribution of weights is close to uniform, a result that may be of independent interest.
\end{abstract}

\section{Introduction}
\subsection{Dedekind's problem in the Boolean lattice}
Dedekind's problem \cite{dedekind} from 1897 is one of the oldest problems in enumerative combinatorics, and in combinatorics more generally. Simply stated, the problem is to determine the number of monotone functions $f:\{0,1\}^n\rightarrow \{0,1\}$, a matter of importance in computer science. Dedekind's problem turns out to be equivalent to a counting problem in the Boolean lattice, which we describe below.

Let $P$ be a partially ordered set, or \emph{poset}. An \emph{antichain} in $P$ is a set of pairwise incomparable elements. Let $\alpha(P)$ denote the size of a largest antichain in $P$, also known as the \emph{width} of $P$, and let $A(P)$ denote the number of antichains in $P$. Since any subset of an antichain is also an antichain, we have the trivial lower bound
\begin{align}\label{eq: trivial lb on A in terms of width} A(P)\geq 2^{\alpha(P)}.\end{align}
The \emph{Boolean lattice} $B_n$ is the poset on the subsets of $\{1,\dots,n\}$, whose partial order is given by the subset relation. As there is a simple bijection between monotone functions $f:\{0,1\}^n\rightarrow \{0,1\}$ and antichains in $B_n$, the answer to Dedekind's problem is precisely $A(B_n)$. But how large is $A(B_n)$?

By Sperner's classical theorem \cite{sperner}, the width of $B_n$ is $\binom{n}{\lfloor n/2\rfloor}$, and the trivial lower bound~\eqref{eq: trivial lb on A in terms of width} gives $A(B_n)\geq \exp_2(\binom{n}{\lfloor n/2\rfloor})$. After seventy years of developments, Kleitman \cite{kleitman69} and subsequently Kleitman and Markowsky \cite{KM75} proved that this simple lower bound is close to optimal, showing that
$$\log_2 A(B_n)\leq \left(1+O\left(\frac{\log n}{n}\right)\right)\cdot\binom{n}{\lfloor n/2\rfloor}.$$
A few years later, Korshunov \cite{korshunov80} determined $A(B_n)$ up to a factor of $(1+o(1))$,  solving Dedekind's problem asymptotically. Such is the theoretical importance of the problem, however, that this was not the end of the story. In 1989, Saphozenko \cite{saphozenko89} gave an alternative proof of Korshunov's bound. Just over a decade later, Kahn~\cite{kahn02} used an entropy argument to give a new proof of the aforementioned bound of Kleitman and Markowsky, while in 2016, Balogh, Treglown and Wagner~\cite{BTW} gave an alternative proof of the same result using the graph container method.

\subsection{Dedekind's problem in the hypergrid}
In this paper, we consider a natural generalization of Dedekind's problem on the \emph{hypergrid}, the poset on $[t]^n$ with the natural coordinate-wise partial order. More precisely, we denote by $[t]$ the set $\{0,\dots,t-1\}$, and for $x,y\in [t]^n$, we write $x\preceq y$ if $x(i)\leq y(i)$ for every $i\in \{1,\dots,n\}$. The Boolean lattice $B_n$ is clearly isomorphic to the poset thus defined on $[2]^n$. For ease of notation, let us write $\alpha(t,n):=\alpha([t]^{n})$ and $A(t,n):=A([t]^n)$. Dedekind's problem in the hypegrid setting is to determine the value of $A(t,n)$

In the past fifteen years, the order of magnitude of $A(t,n)$ has been extensively studied due to its connection to problems in enumerative and extremal combinatorics, which we outline in Subsection \ref{sect:erdos_szekeres} below. The exact value of $A(t,n)$ for arbitrary $t\geq 2$ is only known for $n\in \{1,2,3\}$. Trivially we have $A(t,1)=t+1$, and it is not hard to show that  $A(t,2)=\binom{2t}{t}$ 
(see the proof of Lemma~\ref{lemma:2D}). The case $n=3$ however is highly non-trivial, and is a celebrated theorem  of MacMahon on plane partitions which is discussed in Subsection \ref{sect:erdos_szekeres}.

What about asymptotic estimates of $A(t,n)$ for $t,n\geq 2$? Again, the trivial bound~\eqref{eq: trivial lb on A in terms of width} gives $A(t,n)\geq \exp_2(\alpha(t,n))$, and it is natural to ask whether $A(t,n)$ is close to this lower bound. The asymptotic value of $\alpha(t,n)$ was determined by Anderson~\cite{anderson67}, who showed $$\alpha(t,n)=(1+o(1))\sqrt{\frac{6}{\pi(t^2-1)n}}\cdot t^{n},$$ where $o(1)\rightarrow 0$ as $n\rightarrow\infty$ (independently of $t$). In 2009, Carroll, Cooper, and Tetali \cite{CCT09} proved that $A(t,n)$ is indeed close to the trivial lower bound if $n$ is sufficiently large with respect to $t$. They followed the entropy method of Pippenger \cite{pippenger} to show that if $t\leq n$, then
$$\log_2 A(t,n)\leq \left(1+\frac{11t^2\log t (\log n)^{3/2}}{n^{1/4}}\right)\cdot\alpha(t,n).$$
In the special case $t=3$, Noel, Scott, and Sudakov~\cite{NSS}  used the graph container method to prove   $\log_2 A(3,n)\leq (1+O\left(\sqrt{\frac{\log n}{n}}\right))\cdot \alpha(3,n)$. Subsequently, by a short argument based on chain decompositions, Tsai \cite{tsai} showed that 
$$\log_2 A(t,n)\leq \alpha(t,n)\cdot \log_2 (t+1),$$
which beats the bound of Carroll, Cooper, and Tetali for large $t$. Recently, Pohoata and Zakharov~\cite{PZ21} refined the ideas of \cite{BTW} and \cite{NSS} and used the container method to show that there exists some constant $C_t$ depending only on $t$ such that 
$$\log_2 A(t,n)\leq \left(1+C_t\cdot \frac{\log n}{n^{1/2}}\right)\cdot\alpha(t,n).$$
Even more recently, Park, Sarantis, and Tetali \cite{PST23} employed an entropy based approach to prove that if $t=O(n/\log n)$, then 
$$\log_2 A(t,n)\leq\left(1+ \left(\frac{t(\log n)^3}{n}\right)^{1/2}\right)\cdot\alpha(t,n),$$
and also that $\log_2 A(3,n)\leq (1+O(\frac{1}{n}))\cdot \alpha(3,n)$.

However none of these upper bounds on $\log_2 A(t,n)$ have the right order of magnitude as a function of $t$, and the following simple argument, pointed out by Moshkovitz and Shapira in~\cite{MS12}, beats all of them for $t$ sufficiently large with respect to $n$:  observe the inequality $$A(n,t)\leq A(n-1,t)^{t},$$ which holds since if $A\subset [t]^n$ is an antichain, then for every $i\in [t]$ the slice $A_i=A\cap \{x\in [t]^n:x(n)=i\}$ is isomorphic to an antichain in $[t]^{n-1}$. Therefore, using $A(2,t)=\binom{2t}{t}\leq 4^t$ as the base case, we get by induction that $\log_2 A(n,t)\leq 2t^{n-1}=O(\alpha(n,t)\sqrt{n})$. As an intermediary step in the proof of our main result, we give  a short proof of the sharper bound $\log_2 A(n,t)=O(\alpha(n,t))$ in Theorem~\ref{thm:weak}, which is based on a result of Tomon~\cite{tomon15} about uniform chain decompositions. However, our main result goes even further, and is not only close to optimal for large $t$, but also beats all the previously mentioned bounds for all $n$ sufficiently large.

\begin{theorem}\label{thm:main}
    There exists a constant $c>0$ such that for every $n,t\geq 2$, we have
    $$\log_2 A(t,n)\leq \left(1+c\cdot \frac{(\log n)^3}{n}\right)\cdot \alpha(t,n).$$
\end{theorem}
\noindent We point out that an error term depending on $n$ is necessary. Indeed, in the concluding remarks of this paper we show the lower bound
$$\log_2 A(t,n)\geq (1+2^{-O(n)})\cdot \alpha(t,n).$$
Our proof of Theorem \ref{thm:main} follows in the footsteps of \cite{NSS,PZ21} and is based on the container method. However, we introduce a number of new ideas as well. We give a brief outline of our approach in Section \ref{sect:outline}.

\subsection{High-dimensional partitions and monotone paths}\label{sect:erdos_szekeres}

Let us discuss some corollaries of our main theorem. The study of integer partitions has a long and distinguished history in mathematics, but to keep our discussion concise, we refer an interested reader to~\cite{andrews} as a general reference.  A \emph{$d$-dimensional partition} is a  $d$-dimensional $m\times\dots\times m$-sized non-negative integer array $T$ such that
$$T(a_1,\dots,a_{i},\dots,a_{d})\geq T(a_1,\dots,a_i+1,\dots,a_d)$$
for all indices $a_1,\dots,a_d$ and $i\in \{1,\dots,d\}$. This is also known as a \emph{line partition} (essentially a Young diagram) in the case $d=1$, and as a \emph{plane partition} in the case $d=2$.  Let $P_d(m)$ be the number of $d$-dimensional $m\times \dots\times m$ sized partitions with entries from $\{0,\dots,m\}$. It is a simple exercise to show that $P_1(m)=\binom{2m}{m}$. However, calculating $P_2(m)$ is already highly non-trivial, and it is a celebrated result of MacMahon \cite{macmahon} that
$$P_2(m)=\prod_{1\leq i,j,k\leq m}\frac{i+j+k-1}{i+j+k-2}\approx 2^{1.323\cdot m^2}.$$
Finding similar formulas for $d\geq 3$ is a major open problem in enumerative combinatorics. Nonetheless, as observed by Moshkovitz and Shapira~\cite{MS12}, there is a simple bijection between antichains and high-dimensional partitions, which shows that $P_d(m)=A(m,d+1)$. Therefore, Theorem \ref{thm:main} combined with the estimate on $\alpha(t,n)$ due to Anderson~\cite{anderson67} immediately implies the following.
\begin{theorem}\label{thm:partitions}
    For every $\varepsilon>0$ fixed if $d$ is sufficiently large, then for every $m\geq 2$,
$$(1-\varepsilon) \sqrt{\frac{6}{\pi (m^2-1) d}}\cdot m^{d+1} \leq \log_2 P_d(m)\leq (1+\varepsilon) \sqrt{\frac{6}{\pi (m^2-1) d}}\cdot m^{d+1}.$$
\end{theorem}

Moshkovitz and Shapira~\cite{MS12} proved that the number of high-dimensional partitions is closely related to an extremal problem in Ramsey theory. Given an $r$-uniform hypergraph on the labelled vertex set $\{1,\dots,N\}$, a \emph{monotone path} of length $m$ is a sequence $v_1<\dots<v_{m+r-1}$ such that $\{v_{i},\dots,v_{i+r-1}\}$ is an edge for $i=1,\dots,m$. Let $N_r(q,m)$ be the smallest integer $N$ such that any $q$-coloring of the edges of the complete $r$-uniform hypergraph on the vertex set $\{1,\dots,N\}$ contains a monochromatic path of length $m$. The problem of estimating $N_r(q,m)$ for $r\geq 3$ was proposed by Fox, Pach, Sudakov, and Suk \cite{FPSS}. Their motivation was to find a common generalization of the Erd\H{o}s--Szekeres lemma \cite{ESz} on monotone subsequences, which corresponds to $N_2(2,m)$, and the Erd\H{o}s--Szekeres theorem \cite{ESz} on convex subsets of the plane, which corresponds to $N_3(2,m)$.  In \cite{MS12}, the following surprising identity was established, relating high-dimensional partitions and the Ramsey number of monotone paths:
$$N_3(q,m)=P_{q-1}(m)+1.$$
Therefore, by Theorem \ref{thm:partitions}, we immediately get an almost optimal bound for $N_3(q,m)$.

\begin{theorem}
    For every $\varepsilon>0$ if $q$ is sufficiently large, then for every $m\geq 2$, 
$$(1-\varepsilon)\sqrt{\frac{6}{\pi (m^2-1) q}}\cdot m^{q}\leq \log_2 N_3(q,m)\leq (1+\varepsilon)\sqrt{\frac{6}{\pi (m^2-1) q}}\cdot m^{q}.$$

\end{theorem}

\subsection{Proof outline}\label{sect:outline}

In this section, we give a brief outline of the proof of Theorem \ref{thm:main}. We follow the ideas of \cite{BTW,NSS,PZ21} and use the graph container method to count the number of antichains in $[t]^{n}$. This method was first used by Kleitman and Winston \cite{KW82}, and further developed by Saphozenko \cite{saphozenko05}; however, we do not assume familiarity with the aforementioned papers on the part of the reader.

Our goal is to show that there is a collection $\mathcal{S}$ of  subsets of $[t]^n$, called \emph{fingerprints}, and a function $\psi:\mathcal{S}\rightarrow 2^{[t]^n}$ with the following properties: 
\begin{enumerate}
    \item  every $S\in\mathcal{S}$ is a small antichain; 
    \item for every $S\in \mathcal{S}$, the set $|\psi(S)|$ is at most a bit larger than $\alpha(t,n)$;
\item every antichain of $[t]^n$ is contained in $S\cup \psi(S)$ for some $S\in \mathcal{S}$. 
\end{enumerate}
The elements of the family $\mathcal{C}=\{S\cup \psi(S):S\in\mathcal{S}\}$ are called \emph{containers}. If such a family $\mathcal{S}$ exists, we can upper bound $A(t,n)$ by counting all subsets of $S\cup\psi(S)$ for every $S\in \mathcal{S}$. To ensure that this is small, we need to show that the family $|\mathcal{S}|$ is itself small. We do this by establishing a separate counting result for small antichains (Theorem~\ref{thm:small_antichains}), whose proof is based on a partition of $[t]^n$ into a collection of `rectangles' of uniform sizes. Combined with property 1. above, this will allow us to bound $\vert \mathcal{S}\vert$  as required.

The key idea underpinning the graph container method is that the existence of a family $\mathcal{S}$ as above is guaranteed by a so-called \emph{supersaturation} result. In our setting, this means showing that subsets of $[t]^n$ larger than $\alpha(t,n)$ contain many pairs of comparable elements, and constitutes the most difficult part of the argument. In order to obtain suitable supersaturation results, we construct a \emph{flow} on the cover graph of $[t]^n$ with certain regularity properties --- this is the main work done in Section~\ref{sect: flows}.  We then use this flow in  Section \ref{sect:chain_covers} to define a probability distribution $\phi$ on the maximal chains of $[t]^n$ such that if $\mathbf{C}$ is a random chain with respect to $\phi$, then the probability $\mathbb{P}(x,y\in \mathbf{C})$ is relatively small for all pairs of comparable elements $x,y\in [t]^n$. We then prove the required supersaturation results in Section~\ref{sect:supersaturation}, by analyzing the intersection of a $\phi$-random chain $\mathbf{C}$ with subsets of $[t]^n$. Finally, we weave all the strands of our argument together in Section \ref{sect:containers} to prove our container lemma (Lemma~\ref{lemma:container}) and conclude the proof of Theorem~\ref{thm:main}.

Before we enter the main part of the argument, we gather in Section~\ref{sect:small_antichains} a number of important preliminary results, while in Section~\ref{sect:levels} we present some bounds and estimates on the distribution of the rank sequence (the sequence of sizes of successive levels) of $[t]^n$. The latter results are technical in nature, but they are crucial for showing that the flow we construct in Section~\ref{sect: flows} is sufficiently regular.

\subsection{Notation}
We use standard graph theoretic, poset and Landau notation. When deployed, the $O(.),\Omega(.),\Theta(.)$ notation always hide an absolute constant independent of all our parameters.

A poset $(P,\prec_P)$ is a set $P$ together with a partial ordering $\prec_P$ on $P$. To simplify notation, we often write $P$ instead of $(P,\prec_P)$, and $\prec$ instead of $\prec_P$ if $P$ is clear from the context. An element $y\in P$ \emph{covers} $x\in P$ if $x\prec y$ and there is no $z\in P$ such that $x\prec z\prec y$. We denote this by $x\precdot_P y$, or simply $x\precdot y$. The \emph{cover graph} (also known as the \emph{Hasse diagram}) of $P$ is the graph with vertex set $P$ in which two vertices are joined by an edge if one covers the other. The \emph{comparability graph} of $P$ is the graph on vertex set $P$ in which comparable pairs form the edge set.
An \emph{antichain} in a poset is a set of pairwise incomparable elements, and a \emph{chain} is a set of pairwise comparable elements.

Recall that we write $[t]=\{0,\dots,t-1\}$. Our main object of study in this paper is the $n$-dimensional \emph{hypergrid}, consisting of the set $[t]^{n}$ endowed with the natural coordinate-wise partial ordering: $x\preceq y$ if $x(i)\leq y(i)$ for every $i=1,\dots,n$. Here, as in the rest of the paper, $x(i)$ denotes the $i$-th coordinate of $x$. We write $|x|:=x(1)+\dots+x(n)$. Thus the edges of the cover graph of $[t]^n$ are precisely those pairs $\{x,y\}$ with $|x-y|=1$. We usually write $xy$ as a shorthand for $\{x,y\}$. Furthermore, we use $\alpha(t,n)$ to denote the size of a largest antichain (or \emph{width}) of the poset $[t]^n$, and $A(t,n)$ to denote the total number of antichains in $[t]^n$. 

\section{A weak upper bound and counting small antichains}\label{sect:small_antichains}
In this section, we prove the following upper bound on $A(t,n)$: 
\begin{theorem}\label{thm:weak}
For every $n,t\geq 2$, we have
$$\log_2 A(t,n)=O(\alpha(t,n)).$$ 
\end{theorem}
While this is weaker than the bound promised by Theorem~\ref{thm:main}, we shall use the ideas in the proof of Theorem~\ref{thm:weak} to get a rough count of the number of antichains which are `not too large'. Explicitly, we show the following theorem, which is an essential component of the proof of Theorem~\ref{thm:main}:
\begin{theorem}\label{thm:small_antichains}
For every $n,t,k\geq 2$, the number of antichains of $[t]^{n}$ of size at most $\frac{1}{k}\cdot\alpha(t,n)$ is at most 
$$\exp_2\left(O\left(\frac{\log k}{k}\cdot \alpha(t,n)\right)\right).$$ 
\end{theorem}
\noindent 
We lay the ground for a proof of these two theorems by introducing some important notions related to partial orders. In particular,  we state some of the theoretical tools from the literature on the rank sequences of hypergrids, normalized matchings and chain partitions that we shall need in our argument.

\subsection{Rank sequences, log concavity and width estimates}

Given a poset $P$, we say that $P$  is \emph{graded} if there exists a function $r:P\rightarrow \mathbb{N}$ such that $r(x)=0$ if $x$ is a $\prec_P$-minimal element of $P$, and $r(y)=r(x)+1$ if $y$ covers $x$. If such a function exists, it is unique, and is called the \emph{rank function} of $P$. For $i\in \mathbb{N}$, the \emph{$i$-th level} of $P$ is the set $$L_P(i):=L(i)=\{x\in P:r(x)=i\},$$ and the $i$-th rank number is $N_P(i):=N(i)=|L(i)|$. The \emph{rank sequence} of $P$ is the sequence $N(0),\dots,N(n-1)$, where $n$ is the number of non-empty levels.

A sequence of non-negative real numbers $a_0,a_1,\dots,a_{n-1}$ is \emph{log-concave} if $$(a_i)^2\geq a_{i-1}\cdot a_{i+1}$$ for every $i=1,\dots,n-2$. Also, it is \emph{unimodal} if there exists $m\in [n]$ such that $a_0, a_1, \dots,a_m$ is monotone increasing, and $a_{m},\dots,a_{n-1}$ is monotone decreasing. It is easy to see that if a sequence is log-concave, then it is also unimodal.  Note also that if $i<i'\leq j'< j$ and $i+j=i'+j'$, then log-concavity implies that $$a_ia_j\leq a_i'a_j'.$$

Clearly, $[t]^{n}$ is graded with the $(t-1)n+1$ levels
\begin{align*}
L(i)=\{x\in [t]^n: |x|=i\}, && i\in\{0,1, \ldots, (t-1)n\}.\end{align*}
 It was proved by Anderson~\cite{anderson68} that the rank sequence  of $[t]^{n}$ is log-concave, and thus unimodal as well. 
\begin{lemma}[Anderson \cite{anderson68}]\label{lemma: log-concavity of hypergrid rank sequence}
    The rank sequence of $[t]^{n}$ is log-concave (and thus unimodal).
\end{lemma}
\noindent As $[t]^n$ is symmetric around its middle level, i.e. $N(i)=N((t-1)n-i)$ for $i\in [(t-1)n]$, this implies that a maximum-sized level of $[t]^{n}$ is $L(m)$, where $m=\lfloor (t-1)n/2\rfloor$. By another result of Anderson~\cite{anderson67}, the size of the largest antichain in $[t]^n$ is also $N(m)=\vert L(m)\vert $. Furthermore, Anderson provided the following estimate for the value of $N(m)$:
 \begin{lemma}[Anderson \cite{anderson67}]\label{lemma:middle_level}
 For every fixed $\varepsilon>0$, if $n$ is sufficiently large, then the following holds for every $t\geq 2$: 
$$(1-\varepsilon) \sqrt{\frac{6}{\pi (t^2-1)n}}\cdot t^n\leq \alpha(t,n)=N(m)\leq (1+\varepsilon)\sqrt{\frac{6}{\pi (t^2-1)n}}\cdot t^n.$$
\end{lemma}
In our arguments, we shall only require the weaker result that $\alpha(t,n)=\Theta(t^{n-1}/\sqrt{n})$, a simpler proof of which can be found in~\cite{book_Anderson}.

\subsection{The normalized matching property}

Let $G$ be a bipartite graph with vertex classes $A$ and $B$. For $X\subseteq V(G)$, let $N(X)=\{y\in V(G):\exists x\in X,\ xy\in E(G)\}$ be the \emph{neighborhood of $X$}. Then $G$ has the \emph{normalized matching property} if for every $X\subseteq A$, 
$$\frac{|N(X)|}{|B|}\geq \frac{|X|}{|A|}.$$
Since $N(A\setminus N(Y))\subseteq B\setminus Y$, this also implies that $\frac{|N(Y)|}{|A|}\geq \frac{|Y|}{|B|}$ for every $Y\subseteq B$, so it does not matter how the vertex classes of $G$ are labeled in the definition above. 

If $P$ is a graded poset, we say that $P$ has the \emph{normalized matching property} if the bipartite graph of comparabilities between any two distinct levels has the normalized matching property. Note that it is enough to ensure that this holds for consecutive levels in order to establish the full property. 

\begin{lemma}[Anderson \cite{anderson68} and Harper \cite{harper74}]\label{lemma:normalized_matching}
The poset $[t]^{n}$ has the normalized matching property.
\end{lemma}
Kleitman \cite{kleitman74} proved that the normalized matching property is equivalent to a number of other interesting properties. Given a graded poset $P$ with $n$ levels, $P$ has the \emph{LYM property} if for every antichain $A\subseteq P$, we have 
$$\sum_{i=0}^{n-1}\frac{|L(i)\cap A|}{\vert L(i)\vert}\leq 1.$$
A \emph{maximal chain} in $P$ is a chain that intersects every level of $P$. If $\mathcal{C}$ is the set of maximal chains of $P$, a probability distribution $\phi: \mathcal{C}\rightarrow [0,1]$ is a \emph{regular chain cover} if for every $x\in L(i)$, we have $\mathbb{P}(x\in \mathbf{C})=\frac{1}{\vert L(i)\vert }$, where $\mathbf{C}$ is chosen randomly from $\mathcal{C}$ according to the probability distribution given by $\phi$. 

\begin{lemma}[Kleitman \cite{kleitman74}]\label{lemma:LYM}
The following three properties are equivalent for every graded poset~$P$:
\begin{itemize}
    \item[1.] $P$ has the normalized matching property;
    \item[2.] $P$ has the LYM property;
    \item[3.] $P$ has a regular chain cover.
\end{itemize}
\end{lemma}
The LYM property immediately implies that the width of $P$ is $\max_i N_P(i)$, the size of a largest level in $P$, so we can also deduce from Lemmas~\ref{lemma:normalized_matching} and~\ref{lemma:LYM} that $\alpha(t,n)=N(m)$ with $m=\lfloor (t-1)m/2\rfloor$. We discuss the normalized matching property in more detail in Section \ref{sect:normalized_flow}, where we relate it to the existence of certain normalized flows on the cover graph of $P$.

\subsection{The Cartesian product of posets}

Given posets $(P,\prec_P)$ and $(Q,\prec_Q)$, their \emph{Cartesian product} is the poset $(P\times Q,\prec)$, where $(x,y)\preceq (x',y')$ if $x\preceq_P x'$ and $y\preceq_Q y'$. The Cartesian product preserves a number of useful properties of $P$ and $Q$. If $P$ and $Q$ are graded, then $P\times Q$ is also graded. Also, if $P$ and $Q$ have log-concave rank sequences, then so does $P\times Q$, see \cite{harper74}. Finally, Harper~\cite{harper74} showed if $P$ and $Q$ have log-concave rank sequences, and both have the normalized matching property, then $P\times Q$ also has the normalized matching property. In  Section \ref{sect:normalized_flow}, we discuss Harper's proof in the special case $Q=[t]$.

It is useful to view $[t]^{n}$ as the Cartesian product of $n$ chains of size $t$. In particular, $[t]^n$ is isomorphic to $[t]^m\times [t]^{n-m}$ for every $m\in\{1,\dots,n-1\}$, and we fruitfully identify these two objects in our argument. We also consider the product of two chains, that is, posets isomorphic to $[t_1]\times [t_2]$ for some positive integers $t_1,t_2$; we refer to such posets as \emph{rectangles}. We  notably make use of the following simple lemma bounding the number of antichains contained in a rectangle.
\begin{lemma}\label{lemma:2D}
    Let  $t_1,t_2,u$ be positive integers such that $t_1,t_2\leq u/3$. Then the number of antichains in $[t_1]\times [t_2]$ is at most $4^{u}$. Also, if $\beta\in [0,1/3]$, then the number of antichains of size $\beta u$ in $[t_1]\times [t_2]$ is at most $\exp_2(4\log_2 (1/\beta)\beta u)$.
\end{lemma}

\begin{proof}
For any positive integer $s$, the number of antichains of size $s$ in  $[t_1]\times [t_2]$ is at most the number of antichains of size $s$ in $[u]\times [u]$. Each antichain of size $s$ in $[u]\times [u]$ is uniquely determined by the $s$ distinct rows and columns its elements are contained in, so the number of antichains is at most $\binom{u}{s}^2$. Hence, the number of antichains in $[u]\times [u]$ is 
$$\sum_{s=0}^{u}\binom{u}{s}^2=\binom{2u}{u}\leq 4^u.$$
Also, setting $s=\beta u$ with $\beta\in [0, 1/3]$, we have that the number of antichains of size $\beta u$ in $[t_1]\times [t_2]$ is at most
$$\binom{u}{s}^2\leq (e/\beta)^{2\beta u}= 2^{2\log_2(e/\beta) \beta u}<2^{-4 (\log_2 \beta)\beta u}.$$
\end{proof}

\subsection{Proof of Theorems \ref{thm:weak} and \ref{thm:small_antichains}} 
A classical result of Dilworth~\cite{dilworth50} states that in any finite poset $P$, the width $\alpha(P)$ of $P$ is equal to the size of the smallest chain decomposition of $P$. In order to prove the main theorems of this section, we shall use a partial strengthening of Dilworth's theorem due to the third author~\cite{tomon15}, which states that for a poset $P$ with the normalized matching property and a unimodal rank sequence, one can find a partition of $P$ into $\alpha(P)$ chains, each of which is not much shorter than the average length $\vert P\vert /\alpha(P)$. 
\begin{theorem}[Tomon \cite{tomon15}]\label{thm:chain_partition}
    Let $P$ be a poset with the normalized matching property and a unimodal rank sequence. Then $P$ can be partitioned into $\alpha(P)$ chains, each of which has size at least $\frac{|P|}{2\alpha(P)}-\frac{1}{2}$.
\end{theorem}
\noindent Using this theorem, we show below that $[t]^n$ can be partitioned into large rectangles. This idea has previously been used (albeit in a somewhat more general form) in the context of extremal~\cite{T19} and packing~\cite{T20} problems for posets.
\begin{lemma}\label{lemma:rectangle_partition}
For every $t,n\geq 2$, $[t]^n$ has a partition into sets $R_1,\dots,R_s$ such that $s=\Theta(t^{n-2}/n)$, and each $R_i$ is isomorphic to $[c_i]\times [c_i']$ for some $c_i,c_i'=\Theta(t\sqrt{n})$.
\end{lemma}

\begin{proof}
Let $n_1=\lfloor n/2\rfloor$, $n_2=\lceil n/2\rceil$, and identify the $n$-dimensional hypergrid $[t]^n$ with the Cartesian product $[t]^{n_1}\times [t]^{n_2}$. Recall that by Lemmas~\ref{lemma: log-concavity of hypergrid rank sequence} and~\ref{lemma:normalized_matching}, the hypergrids $[t]^{n_1}$ and $[t]^{n_2}$ have unimodal rank sequences and the normalized matching property. We may thus apply Theorem~\ref{thm:chain_partition}: writing  $w_1=\alpha(t,n_1)$, there exists a partition of $[t]^{n_1}$ into chains $D_1,\dots,D_{w_1}$, each of which has length at least $t^{n_1}/(2w_1) -1/2$. If $|D_j|> t^{n_1}/w_1$, partition $D_j$ into subchains, each of which has size between $\frac{t^{n_1}}{2w_1}-\frac{1}{2}$ and $t^{n_1}/w_1$. In this way, we obtain a partition of $[t]^{n_1}$ into chains $C_1,\dots,C_{k_1}$ with some $k_1$ such that for $j\in \{1,\dots,k_1\}$,
$$|C_j|=\Theta\left(\frac{t^{n_1}}{w_1}\right)=\Theta\left(t\sqrt{n_1}\right)=\Theta\left(t\sqrt{n}\right),$$
where the second equality holds by Lemma \ref{lemma:middle_level}. This implies in particular that $k_1=\Theta(t^{n_1-1}/\sqrt{n})$. Similarly, we can partition $[t]^{n_2}$ into chains $C_1',\dots,C_{k_2}'$, each of size $\Theta(t\sqrt{n})$, where $k_2=\Theta(t^{n_2-1}/\sqrt{n})$. The rectangular sets $C_j\times C_{j'}$ then form the desired partition for $(j,j')\in \{1,\dots,k_1\}\times\{1,\dots,k_2\}$, and we have $s=k_1k_2=\Theta(t^{n-2}/n)$ of them in total, as desired.
\end{proof}
\noindent We refer to the partitions of $[t]^n$ into rectangles whose existence is guaranteed by Lemma~\ref{lemma:rectangle_partition} as \emph{good rectangular partitions}.
\begin{proof}[Proof of Theorem \ref{thm:weak}]
Let $R_1,\dots,R_s$ be a good rectangular partition of $[t]^n$. Let $A_{j}$ be the number of antichains of $R_{j}$. Then $A(t,n)\leq \prod_{j=1}^{s}A_j.$ But $R_j$ is isomorphic to $[c_j]\times [c_j']$ for some $c_j,c_j'=O(t\sqrt{n})$, so $A_{j}=2^{O(t\sqrt{n})}$ by Lemma \ref{lemma:2D}. Therefore, 
$$\log_2 A(t,n)=O(s\cdot (t\sqrt{n}))=O\left(\frac{t^{n-1}}{\sqrt{n}}\right)=O(\alpha(t,n)).$$
\end{proof}

\begin{proof}[Proof of Theorem \ref{thm:small_antichains}]
Consider two cases depending on $k$. If $k\geq n^{1/4}t^{1/2}$, then we can bound the number of antichains of size at most  $s=\lfloor \alpha(t,n)/k\rfloor$ by counting all subsets of $[t]^{n}$ of size at most $s$. Thus, we get the upper bound
$$\sum_{i=0}^{s}\binom{t^{n}}{i}\leq 2\binom{t^n}{s}\leq 2\left(\frac{et^n}{s}\right)^{s}\leq (O(t\sqrt{n}k))^{s}=\exp_2(O(s\log k)),$$
where the third inequality holds by $\alpha(t,n)=\Theta(t^{n-1}/\sqrt{n})$, and the last equality holds by our lower bound on $k$. Hence, in this case we are done.

In what follows, assume that $k\leq  n^{1/4}t^{1/2}$. Let $R_1,\dots,R_{s}$ be a good rectangular partition of $[t]^n$. Then each $R_j$ is isomorphic to $[c_j]\times [c_j']$ with $c_j,c_j'=O(t\sqrt{n})$. Set $u=3 \max\{c_1,\dots,c_s,c_1',\dots,c_s'\}$, and note that $u=O(t\sqrt{n})$. Let $r\leq \alpha(t,n)/k$ be an integer, and let $r=\sum_{j=1}^{s}r_{j}$ be a partition of $r$. Let us count the number of antichains $A\subset [t]^n$ such that $|A\cap R_j|=r_{j}$. By Lemma \ref{lemma:2D}, we get that this number is at most 
\begin{equation}\label{equ:thm_antichains}
    \exp_2\left(-4u\sum_{j=1}^s\frac{r_j}{u}\log_2 \frac{r_j}{u}\right).
\end{equation}
As the function $f(x)=-x\log_2 x$ is concave, and $\sum_{j=1}^s \frac{r_j}{u}=\frac{r}{u}$, we get by Jensen's inequality that (\ref{equ:thm_antichains}) is at most
$$\exp_2\left(-4us\cdot \frac{r}{s\cdot u}\log_2\frac{r}{s\cdot u}\right)=\exp_2\left(4r\log_2\frac{s\cdot u}{r}\right)=\exp_2\left(O\left(\frac{\alpha(t,n)}{k}\log_2 k\right)\right).$$
Here, in the last equality, we used that $s=O(t^{n-2}/n)$, $u=O(t\sqrt{n})$, and so $su=O(t^{n-1}/\sqrt{n})=O(\alpha(t,n))$. 

Note that there are at most $\alpha(t,n)/k$ choices for $r$, and that the number of partitions of $r$ into $s$ non-negative integers is at most 
\begin{align*}
\binom{r+s-1}{s-1}  <  \binom{t^n}{s}\leq& \left(\frac{e t^{n}}{s}\right)^{s}\leq (O(t^2 n))^{O(t^{n-2}/n)}=\exp_2\left(O\left(\frac{(\log t+\log n)t^{n-2}}{n}\right)\right)\\
    =&\exp_2\left(O\left(\alpha(t,n)\cdot \frac{\log(t)+\log (n) }{t\sqrt{n}}\right)\right) =\exp_2\left(O\left(\frac{\alpha(t,n)}{k}\log_2 k\right)\right).
\end{align*}
Here, the last equality holds by our upper bound on $k$.  This finishes the proof.
\end{proof}

In the rest of our paper, we may (and will implicitly) assume that $n$ is sufficiently large.  We can deal with small $n$ using Theorem \ref{thm:weak} if we choose $c$ to be sufficiently large in Theorem \ref{thm:main}.

\section{Distribution of levels in the hypergrid}\label{sect:levels}

In this section, we present a number of bounds about the distribution of the level sizes of $[t]^n$. Many of these bounds are rather technical in nature and do not contribute much to the main ideas of the paper, and a reader may safely skip over this part on their first reading. 

Let $N(0),\dots,N((t-1)n)$ denote the sizes of the levels of $[t]^n$.
\begin{lemma}\label{lemma:level_upperbound}
For every integer $r$ such that  $|r|>t$,
$$N\left(\frac{(t-1)n}{2}-r\right)\leq t^{n-1}\cdot e^{-r^2/(2t^2(n-1))}.$$
\end{lemma}

\begin{proof}
As $N\left(\frac{(t-1)n}{2}+r\right)= N\left(\frac{(t-1)n}{2}-r\right)$, it is enough to consider the case $r>t$. Let $N'(0),\dots,N'((t-1)(n-1))$ be the sizes of the levels of $[t]^{n-1}$, and set $N'(i)=0$ if $i<0$ or $i>(t-1)(n-1)$. Letting $m=(t-1)n/2$ and considering the possible value of the last coordinate of a vector $x\in L(i)$ , we then have 
$$N(m-r)=\sum_{i=m-r-t+1}^{m-r}N'(i)\leq  \sum_{i=0}^{m-r}N'(i).$$
Let $X_1$ be a random element of $[t]$ taken from the uniform distribution, and let $X_2,\dots,X_{n-1}$ be independent copies of $X_1$. Set $X=X_1+\dots+X_{n-1}$, then
$$\sum_{i=0}^{m-r}N'(i)=t^{n-1}\cdot \mathbb{P}\left(X\leq \mathbb{E}(X)+\frac{t-1}{2}-r\right)\leq t^{n-1}\cdot \mathbb{P}\left(X\leq \mathbb{E}(X)-\frac{r}{2}\right).$$
By Hoeffding's inequality, the probability on the right-hand-side is at most $e^{-r^2/(2t^2(n-1))}$, finishing the proof.
\end{proof}
\noindent  The remainder of this section is devoted to the proof of the following technical result, which will be needed to estimate certain flows on $[t]^n$ later.
\begin{lemma} \label{lemma:a+b=c+d bound}
Let $k$ be an integer such that $|k-(t-1)n/2|\leq tn^{2/3}$. Let $\Lambda$ be the maximum of $N\left(z_1\right)N\left(z_2\right)-N\left(z_3\right)N\left(z_4\right)$
among all integers $z_1, z_2, z_3, z_4$ satisfying  $k-t+1\leq z_1,z_2,z_3,z_4\leq k+1$ and $z_1+z_2=z_3+z_4$, and
let $N_{min}$ be the minimum of $N\left(k-t+1\right),\dots,N\left(k+1\right)$.
Then $$\frac{\Lambda}{N_{min}^{2}}=O\left(\frac{1}{n}\right).$$
\end{lemma}
The main idea in the proof of this lemma is to define an integer-valued random variable $X$  so that the distribution of $X$  `captures' the distribution of the sizes of the levels of $[t]^n$ in a certain sense, and so that most of the weight of $X$'s distribution is concentrated around $k$. We then use tools from the study of \emph{local limit theorems} (see e.g.\  \cite[Section 3.5]{durrett}) and analytic techniques to estimate quantities of the form $$\mathbb{P}(X=a_1)\cdot \mathbb{P}(X=a_2)-\mathbb{P}(X=a_3)\cdot \mathbb{P}(X=a_4).$$


Having outlined the strategy, we now delve into the details. Let $k$ be an integer with $k\in\left[\frac{n\left(t-1\right)}{2}-tn^{2/3},\frac{n\left(t-1\right)}{2} + tn^{2/3}\right]$, and set $q=\frac{n\left(t-1\right)}{2}-k$. Given a real number $p\geq 0$, let $X_1=X_1(p), X_2, \dots,X_n$ be i.i.d.\ random variables with distribution 
 \begin{align}\label{eq:def_of_X_i}
\mathbb{P}\left(X_{1}=s\right)=\alpha^{-1}\cdot p^s\ \ \ \mbox{ for }s=0,\dots,t-1,
\end{align}
 where $\alpha=\alpha(p):=\sum_{j=0}^{t-1}p^j$ is the normalizing factor. The expectation of $X_1$ is
$$\mu(p):=\mathbb{E}(X_1)=\frac{1}{\alpha(p)}\cdot \sum_{j=0}^{t-1}j\cdot p^{j}.$$
It is a simple exercise to verify that $\mu(p)$ is continuous, strictly monotone increasing in $\mathbb{R}_{\geq 0}$ and satisfies $\mu(0)=0$, $\mu(1)=\frac{t-1}{2}$ and $\mu(p)\rightarrow t-1$ as $p\rightarrow \infty$. In particular, there exists a unique $p=p(k,n)\geq 0$ such that $\mu(p)=\frac{k}{n}$ is satisfied. Finally, we define $X$ by setting 
\begin{align}\label{eq:def_of_X}
X=\left(\sum_{j=1}^{n}X_{j}\right)-k,
\end{align}
and note that $\mathbb{E}(X)=0$. A key observation is that the distribution of $X$ is closely related to that of the levels of $[t]^n$. Indeed, we have $$\mathbb{P}(X=s-k)=\alpha^{-n}\cdot p^s\cdot N(s).$$ 
As a first step in our proof of Lemma~\ref{lemma:a+b=c+d bound}, we show that $p=p(k,n)$ is close to $1$. 
\begin{lemma}\label{lemma:p}
For $n$ sufficiently large, we have 
$$\vert 1-p\vert = O\left(\frac{|q|}{t^{2}n}\right).$$ 
\end{lemma}

\begin{proof}
 First, consider the case when  $q \geq 0$. As $\mu(1)=\frac{t-1}{2}$, the monotonicity of $\mu$ implies that we must have $p \leq 1$, and hence it suffices to prove that $1-p=O(|q|/t^2n)$. Let $x\in [0,1]$, then Bernoulli's inequality implies that for any positive integer $j$,
$$(1-x)^j \geq 1-jx$$
and by a simple generalization of Bernoulli's inequality, 
$$(1-x)^j \leq 1 - jx + \frac{j(j-1)}{2}x^2.$$
Thus, it follows that  
$$ \mu(1-x) \leq \frac{\sum_{j=0}^{t-1} \left(j - j^2x + \frac{j^2(j-1)}{2}x^2\right)}{\sum_{j=0}^{t-1} \left(1 - jx\right)}    
    = \frac{t-1}{2}\cdot \left(1 - \frac{(t+1)x}{6} + O(x^2t^2)\right),$$
where we are using the fact that $\frac{1}{1-\frac{t-1}{2}x} = 1 + \frac{t-1}{2}x + O(x^2t^2)$. Set $x=100q/(t^2n)$. Then $xt \rightarrow 0$ as $n \rightarrow \infty$, and thus we get $\mu(1-x)\leq \frac{t-1}{2}-\frac{k}{n}=\mu(p)$ for sufficiently large $n$. By the monotonicity of $\mu$, we must have $1-p\leq x = O\left(\frac{q}{t^2n}\right)$. The case when $q < 0$ follows in a similar manner. 
\end{proof}

The characteristic function of the random variable $X_{1}-\frac{k}{n}$
is given by the function $$\varphi\left(y\right):=\mathbb{E}e^{iy(X_1-\frac{k}{n})}= \frac{1}{\alpha}\cdot \sum_{j=0}^{t-1} p^je^{iy\left(t-\frac{k}{n}\right)},$$
and the characteristic function of $X$ is then given by the $n$-fold product $\varphi\left(y\right)^{n}$.
Define $f:\mathbb{C}\rightarrow\mathbb{C}$ by setting
\begin{align}\label{eq: f-def}
f\left(x\right):=\frac{1}{2\pi}\int_{-\pi}^{\pi}e^{-ixy}\varphi\left(y\right)^{n}dy.
\end{align}
Then $f$ is clearly holomorphic, and by the Fourier inverse theorem, we
have $$\mathbb{P}\left(X=m\right)=f\left(m\right)$$ for every $m$ integer. Furthermore, the restriction of $f$ to $\mathbb{R}$ is a real-valued function, infinitely differentiable as a real-valued function, and its real derivative coincides with its complex derivative. Thus, we have 
\[
f'\left(x\right)=\frac{1}{2\pi}\int_{-\pi}^{\pi}-iye^{-ixy}\varphi\left(y\right)^{n}dy
\]
 and 

\[
f''\left(x\right)=\frac{1}{2\pi}\int_{-\pi}^{\pi}-y^{2}e^{-ixy}\varphi\left(y\right)^{n}dy.
\]
Recall that our aim is to bound expressions of the form $N\left(z_1\right)N\left(z_2\right)-N\left(z_3\right)N\left(z_4\right)$,
where $z_1+z_2=z_3+z_4$ and $z_1,z_2,z_3, z_4$ are all close to $k$. We shall do this by relating the quantities $N(z_i)$ to the values of $f$ at the points $a_i=z_i-k$, and using Taylor's approximation theorem to bound the size of the difference $f(a_1)f(a_2)-f(a_3)f(a_4)$. To do this, we will need some bounds on the magnitude of $f$, $f'$ and $f''$ near $0$, which are given in Lemma~\ref{lemma:derivatives} below. Recall that for a function $g$, the \emph{supremum norm} of $g$ is defined as
$$\left\Vert g\right\Vert _{\infty}:=\sup_{x\in\mathbb{R}}\left|g\left(x\right)\right|.$$

\begin{lemma} \label{lemma:derivatives}
We have  
\begin{enumerate}
\item $f(x)=\Theta\left(\frac{1}{t\sqrt{n}}\right)$ for $x\in\left[-t,t\right]$,
\item $\left\Vert f'\right\Vert _{\infty}=O\left(\frac{1}{t^{2}n}\right)$,
\item $\left\Vert f''\right\Vert _{\infty}=O\left(\frac{1}{t^{3}n^{3/2}}\right)$. 
\end{enumerate}
\end{lemma}

\begin{proof}
See the Appendix.
\end{proof}

In the next Lemma, we use Taylor's theorem to bound a difference of
the form $f\left(a_{1}\right)f\left(a_{2}\right)-f\left(a_{3}\right)f\left(a_{4}\right)$
where $a_{1}+a_{2}=a_{3}+a_{4}$ by using the estimates obtained in
the previous lemma. 
\begin{lemma}\label{lemma:a_1234}
Let $a_{1},\dots,a_{4}\in \left[-t,t\right]$ such that $a_{1}+a_{2}=a_{3}+a_{4}$. Then $$\left|f\left(a_{1}\right)f\left(a_{2}\right)-f\left(a_{3}\right)f\left(a_{4}\right)\right|=O\left(\frac{1}{n^2t^2}\right).$$
\end{lemma}

\begin{proof}
By Taylor's theorem, there exists $\xi_{i}\in\left[-t,t\right]$
for $i\in\left\{ 1,\dots,4\right\} $ so that $$f\left(a_{i}\right)=f\left(0\right)+a_{i}f'\left(0\right)+\frac{1}{2}a_{i}^{2}f''\left(\xi_{i}\right).$$
Using the condition $a_{1}+a_{2}=a_{3}+a_{4}$, we can write
\begin{align*}
f(a_{1})f(a_{2})-f(a_{3})&f(a_{4}) =(a_{1}a_{2}-a_{3}a_{4})f'(0)^{2}\\
 &+\frac{1}{2}f(0)\cdot(a_{1}^{2}f''(\xi_{1})+a_2^2f''(\xi_2)-a_3^2f''(\xi_3)-a_4^2f''(\xi_4))\\
 & +\frac{1}{2}f'(0)\cdot (a_{1}^{2}a_{2}f''(\xi_{1})+a_{1}a_{2}^{2}f''(\xi_{2})-a_{3}^{2}a_{4}f''(\xi_{3})-a_{3}a_{4}^{2}f''(\xi_{4}))\\
 & +\frac{1}{4}(a_{1}^{2}a_{2}^{2}f''(\xi_{1})f''(\xi_{2})-a_{3}^{2}a_{4}^{2}f''(\xi_{3})f''(\xi_{4})).
\end{align*}
Bound each term in the right-hand-side by its absolute value. Then, using $\left|a_{i}\right|\leq t$ for each $i$, 
\begin{align*}
    |f\left(a_{1}\right)f\left(a_{2}\right)&-f\left(a_{3}\right)f\left(a_{4}\right)|\\
    &\leq  O\left(t^{2}\left\Vert f'\right\Vert _{\infty}^{2}+t^{2}\left|f\left(0\right)\right|\left\Vert f''\right\Vert _{\infty}+t^{3}\left\Vert f'\right\Vert _{\infty}\left\Vert f''\right\Vert _{\infty}+t^{4}\left\Vert f''\right\Vert _{\infty}^{2}\right).
\end{align*}
Thus, Lemma \ref{lemma:derivatives} implies that $\left|f\left(a_{1}\right)f\left(a_{2}\right)-f\left(a_{3}\right)f\left(a_{4}\right)\right|=O\left(\frac{1}{t^{2}n^{2}}\right).$
\end{proof}

\begin{proof}[Proof of Lemma \ref{lemma:a+b=c+d bound}]
Let $z_1, z_2, z_3, z_4$ and $m$ be integers in $\left[k-t+1,k+1\right]$ with $z_1+z_2=z_3+z_4$ such that  $\Lambda=N\left(z_1\right)N\left(z_2\right)-N\left(z_3\right)N\left(z_4\right)$
and $N_{min}=N\left(m\right)$. Let $a_i=z_i-k$ for $i\in [z]$, let $X$ be the random variable and
$f$ be the function defined in~\eqref{eq:def_of_X} and~\eqref{eq: f-def} respectively.
 Recall that for every $s\in[(t-1)n]$, we have
\[
f\left(s-k\right)=\mathbb{P}\left(X=s-k\right)=\alpha^{-n}p^{s}N\left(s\right).
\]
Therefore, $$\Lambda=\alpha^{2n}p^{-\left(z_1+z_2\right)}\left(f\left(a_{1}\right)f\left(a_{2}\right)-f\left(a_{3}\right)f\left(a_{4}\right)\right),$$
and 
$$N_{min}^2=\alpha^{2n}p^{-2m}f\left(m-k\right)^{2}.$$
By Lemma \ref{lemma:p}, we have $|1 - p| = O\left(\frac{|q|}{t^{2}n}\right)$.   Since $\left|2m-\left(z_1+z_2\right)\right|\leq2t$ and  $q\leq tn^{2/3}$, it follows that for sufficiently large $n$,
$$ p^{2m-\left(z_1+z_2\right)}\leq e^{O\left(\frac{|q|}{tn}\right)} = O(1)$$ by using the estimate $1 + x \leq e^{x}$.
Thus, it follows from  1. in Lemma \ref{lemma:derivatives} (noting that $|m-k|\leq t$) and Lemma \ref{lemma:a_1234} that
\[
\frac{\Lambda}{N_{min}^{2}}=p^{2m-(z_1+z_2)}\cdot \frac{f(a_1)f(a_2)-f(a_3)f(a_4)}{f(m-k)^2}= O\left(\frac{1}{n}\right),
\]
as required. 
\end{proof}
\section{Normalized flows}\label{sect: flows}
\subsection{Normalized flows in Cartesian products}\label{sect:normalized_flow}
In this subsection, we discuss flows on graphs and investigate the normalized matching property in more detail. In particular, we use a method of Harper for constructing normalized flows in Cartesian products of posets to obtain a normalized flow in the hypergrid.

A \emph{weighted graph} is a pair $(G,\sigma)$, where $G$ is a graph, and $\sigma:V(G)\rightarrow \mathbb{R}_{\geq 0}$. A \emph{flow} on $(G,\sigma)$ is a function $f:E(G)\rightarrow \mathbb{R}_{\geq 0}$ such that $$\sigma(v)=\sum_{w: vw\in E(G)}f(vw)$$ for every $v\in V(G)$. If $G$ is a bipartite graph with vertex classes $A$ and $B$, the \emph{canonical weighting} of $G$ is defined as $\sigma_c(v)=\frac{1}{|A|}$ for every $v\in A$, and $\sigma_c(v)=\frac{1}{|B|}$ for every $v\in B$. We refer to a flow of $(G,\sigma_c)$ as a \emph{normalized matching flow} of $G$. A result of Kleitman relates the normalized matching property to the existence of normalized matching flows.
\begin{lemma}[Kleitman \cite{kleitman74}]\label{lemma:flow} A bipartite graph $G$ has the normalized matching property if and only if it has a normalized matching flow.
\end{lemma}
It turns out to be convenient to work with a scaled version of the canonical weighting. Breaking the symmetry, define the  \emph{scaled canonical weighting} of a bipartite graph $G$ with classes labelled $A$ and $B$ as $\sigma_s(v)=1$ if $v\in A$, and $\sigma_s(v)=|A|/|B|$ if $v\in B$. A \emph{scaled normalized matching flow} (or SNMF for short) of $G$ is then a flow of $(G,\sigma_s)$, and we observe that $G$ has an SNMF if and only if it has a normalized matching flow.
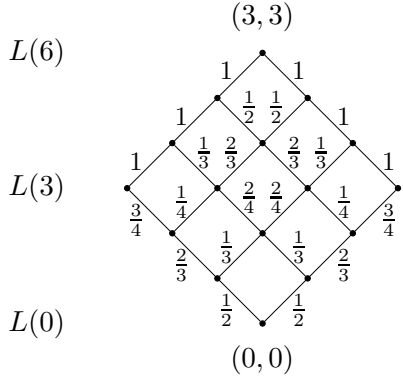
\begin{figure}
\begin{center}
\begin{tikzpicture}[scale=1.2]
    \node at (-2.5,0) {$L(0)$};
    \node[vertex,minimum size=2pt] (v00) at (0,0)  {}; \node at (0,-0.4) {$(0,0)$};
    \node at (-0.4,0.15) {\small $\frac{1}{2}$};
    \node at (0.4,0.15) {\small $\frac{1}{2}$};

    \node[vertex,minimum size=2pt] (v10) at (-0.5,0.5)  {};
    \node[vertex,minimum size=2pt] (v01) at (0.5,0.5)  {};
    \node at (-0.9,0.65) {\small $\frac{2}{3}$};
    \node at (-0.4,0.85) {\small $\frac{1}{3}$};
    \node at (0.9,0.65) {\small $\frac{2}{3}$};
    \node at (0.4,0.85) {\small $\frac{1}{3}$};
    
    \node[vertex,minimum size=2pt] (v20) at (-1,1)  {};
    \node[vertex,minimum size=2pt] (v11) at (0,1)  {};
    \node[vertex,minimum size=2pt] (v02) at (1,1)  {};
    \node at (-1.4,1.15) {\small $\frac{3}{4}$};
    \node at (-0.9,1.35) {\small $\frac{1}{4}$};
    \node at (-0.15,1.4) {\small $\frac{2}{4}$};

    \node at (1.4,1.15) {\small $\frac{3}{4}$};
    \node at (0.9,1.35) {\small $\frac{1}{4}$};
    \node at (0.15,1.4) {\small $\frac{2}{4}$};

    \node at (-2.5,1.5) {$L(3)$};
    \node[vertex,minimum size=2pt] (v30) at (-1.5,1.5)  {};
    \node[vertex,minimum size=2pt] (v21) at (-0.5,1.5)  {};
    \node[vertex,minimum size=2pt] (v12) at (0.5,1.5)  {};
    \node[vertex,minimum size=2pt] (v03) at (1.5,1.5)  {};
    
     \node at (-1.4,1.8) {\small $1$};
    \node at (-0.65,1.9) {\small $\frac{1}{3}$};
    \node at (-0.35,1.9) {\small $\frac{2}{3}$};

     \node at (1.4,1.8) {\small $1$};
    \node at (0.65,1.9) {\small $\frac{1}{3}$};
    \node at (0.35,1.9) {\small $\frac{2}{3}$};

    \node[vertex,minimum size=2pt] (v31) at (-1,2)  {};
    \node[vertex,minimum size=2pt] (v22) at (0,2)  {};
    \node[vertex,minimum size=2pt] (v13) at (1,2)  {};

    \node at (-0.9,2.3) {\small $1$};
    \node at (-0.15,2.4) {\small $\frac{1}{2}$};
    \node at (0.9,2.3) {\small $1$};
    \node at (0.15,2.4) {\small $\frac{1}{2}$};
    
    \node[vertex,minimum size=2pt] (v32) at (-0.5,2.5)  {};
    \node[vertex,minimum size=2pt] (v23) at (0.5,2.5)  {};

     \node at (-0.4,2.8) {\small $1$};
     \node at (0.4,2.8) {\small $1$};

    \node at (-2.5,3) {$L(6)$};
    \node[vertex,minimum size=2pt] (v33) at (0,3)  {}; \node at (0,3.4) {$(3,3)$};

    \draw[very thin] (v00) -- (v30) ;
    \draw[very thin] (v01) -- (v31) ;
    \draw[very thin] (v02) -- (v32) ;
    \draw[very thin] (v03) -- (v33) ;
    
    \draw[very thin] (v00) -- (v03) ;
    \draw[very thin] (v10) -- (v13) ;
    \draw[very thin] (v20) -- (v23) ;
    \draw[very thin] (v30) -- (v33) ;
    
\end{tikzpicture}
\caption{A scaled normalized matching flow (SNMF) of $[4]^2$.}
\label{fig:SNMF}
\end{center}
\end{figure}

Let $P$ be a graded poset, and write $E(P)$ for the edges of the cover graph of $P$. Writing $n$ for the number of levels of $P$, we say that a function $f: E(P)\rightarrow \mathbb{R}_{\geq 0}$ is a SNMF of $P$ if for $i=0,\dots,n-2$, the restriction of $f$ to the bipartite graph induced by the levels $L(i)\cup L(i+1)$ is an SNMF of that graph with $L(i)$ playing the role of the distinguished part $A$. See Figure \ref{fig:SNMF} for an example. By our earlier observation that a graded poset has the normalized matching property if and only consecutive levels have the normalized matching property, Lemma \ref{lemma:flow} implies that the poset $P$ has the normalized matching property if and only if it has an SNMF.

In the remainder of this subsection, our goal is to show that there exists an SNMF $f$ of $[t]^{n}$ such that for all edges $uv$ of the cover graph, $f(uv)$ is not too large. In order to construct such an SNMF of $[t]^n$, we follow a method of Harper~\cite{harper74}. Given two graded posets $P$ and $Q$ with log-concave rank-sequences, and given SNMF $f_P$ of $P$ and $f_Q$ of $Q$, Harper~\cite{harper74} showed how to construct an SNMF of the Cartesian product $P\times Q$. We present his method in the special case when $Q=[t]$, the chain of length $t$.

Let $P$ be a graded poset with $n$ levels and a log-concave rank sequence, and let $f_P$ be an SNMF of $P$. For ease of notation, we write $L_P(i)=\emptyset$ if $i<0$ or $i\geq n$. Setting $N_P(i)=\vert L_P(i)\vert =0$ for $i<0$ or $i\geq n$, we note that the infinite sequence $\dots,N_P(-1),N_P(0),N_P(1),\dots$ remains log-concave. Now the poset $R=P\times [t]$ is graded, and we have $$L_R(i)=\bigcup_{j=0}^{t-1}L_P(i-j)\times \{j\}$$ for every $i$. The cover graph of $R$ has two types of edges. If the pair $\{(x,i),(y,j)\}$ is an edge with $x,y\in P$, $i,j\in [t]$, then we say that it is a \emph{left edge} if $i=j$, and that it is a \emph{right edge} if $x=y$ (and $\vert i-j\vert=1$).

Now let $k$ be an integer such that $L_R(k)$ and $L_R(k+1)$ are non-empty, and consider the bipartite graph $G$ between $L_R(k)$ and $L_R(k+1)$. By collapsing the levels of $P$, we define a weighted bipartite graph $(H,\sigma)$ as follows. The vertex classes of $H$ are $A=\{a_0,\dots,a_{t-1}\}$ and $B=\{b_0,\dots,b_{t-1}\}$, and $a_i$ and $b_j$ are joined by an edge if $i=j$ or $i=j-1$. Set $\sigma(a_i)=N_P(k-i)$, and set $\sigma(b_i)=\frac{N_R(k)}{N_R(k+1)}\cdot N_P(k+1-i)$. One can think of $a_i$ as a collapse of the set $L_P(k-i)\times\{i\}$ to a single point, and similarly of $b_i$ as a collapse of the set $L_P(k+1-i)\times \{i\}$. See Figure \ref{figure:1} for an illustration. 

Suppose there exists a flow $g$ of $(H,\sigma)$. We can then define a function $f:E(G)\rightarrow \mathbb{R}_{\geq 0}$ as follows. If $e=\{(x,i),(y,i)\}$ is a left edge for some $x\in L_P(k-i)$, $y\in L_P(k+1-i)$, $i\in [t]$, we set $$f(e)=\frac{1}{N_P(k-i)}\cdot f_P(xy)\cdot g(a_ib_i).$$
On the other hand, if $e=\{(x,i),(x,i+1)\}$ is a right edge for some $x\in L_P(k-i)$, $i\in [t-1]$, we set $$f(e)=\frac{1}{N_P(k-i)}\cdot g(a_ib_{i+1}).$$
Note that if $e$ is a left edge, then using that $g$ is a flow, we have $g(a_ib_i)\leq \sigma(a_i)=N_P(k-i)$, and hence $f(e)\leq f_P(xy)$. Similarly, if $e$ is a right edge, then $g(a_ib_{i+1})\leq N_P(k-i)$ and $f(e)\leq 1$.

\begin{figure}
\begin{center}
    \begin{tikzpicture}
        \node[vertex,label=below:$a_0$] (a0) at (-4,-1) {}; \node at (-4,-2) {$N_P(k)$};
        \node[vertex,label=above:$b_0$] (b0) at (-5,1) {}; \node at (-5,2) {\small $\frac{N_R(k)}{N_R(k+1)}\cdot N_P(k+1)$};
        \node[vertex,label=below:$a_1$] (a1) at (-2,-1) {}; \node at (-2,-2) {$N_P(k-1)$};
		\node[vertex,label=above:$b_1$] (b1) at (-3,1) {};
        \node[vertex,label=above:$b_2$] (b2) at (-1,1) {};

        \node[vertex,label=above:$b_{t-2}$] (bt-) at (2,1) {};
        \node[vertex,label=below:$a_{t-2}$] (at-) at (3,-1) {}; 
        \node[vertex,label=below:$a_{t-1}$] (at) at (5,-1) {}; \node at (5,-2) {$N_P(k-t+1)$};
		\node[vertex,label=above:$b_{t-1}$] (bt) at (4,1) {};  \node at (4,2) {\small $\frac{N_R(k)}{N_R(k+1)}\cdot N_P(k-t+2)$};
        \node at (0.5,0) {$\dots$};
        \node at (-7,-1) {$A$};
        \node at (-7,1) {$B$};
 
        \draw (a0) edge (b0) ;
        \draw (a0) edge (b1) ;
        \draw (a1) edge (b1) ;
        \draw (a1) edge (b2) ;
        \draw (b2) edge (-0.5,0) ;

        \draw (1.5,0) edge (bt-);
        \draw (bt-) edge (at-);
        \draw (at-) edge (bt) ;
        \draw (bt) edge (at) ;

    \end{tikzpicture}
\end{center}
\caption{An illustration of the weighted bipartite graph $(H,\sigma)$.}
\label{figure:1}
\end{figure}
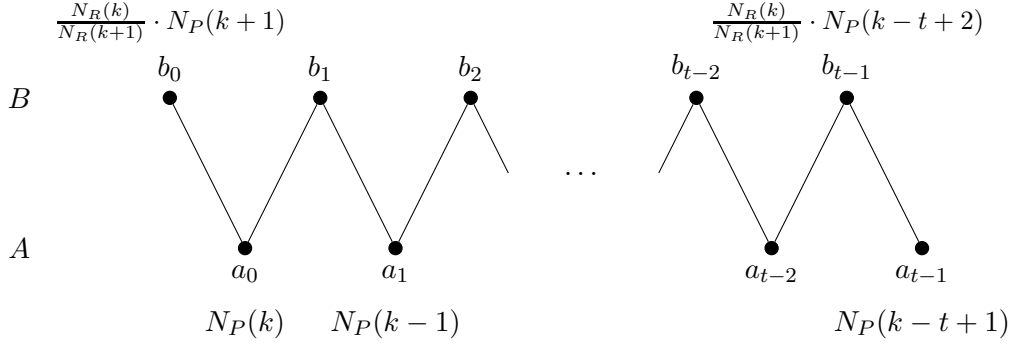

\begin{claim}\label{claim: f is an SNMF}
$f$ is an SNMF of $G$.
\end{claim}

\begin{proof}
Let $(x,i)\in L_R(k)$ with $i\in [t]$ and $x\in L_P(k-i)$. Then
\begin{align*}
    \sum_{(y,j)\sim_G (x,i)}f((x,i)(y,j))&=f((x,i)(x,i+1))+\sum_{x\precdot y}f_P((x,i)(y,i))\\
    &=\frac{1}{N_P(k-i)}\cdot g(a_ib_{i+1})+\sum_{x\precdot y}\frac{1}{N_P(k-i)}\cdot f_P(xy)\cdot g(a_ib_i)\\
    &=\frac{1}{N_P(k-i)}\cdot g(a_i b_{i+1})+\frac{1}{N_P(k-i)}\cdot g(a_ib_i)=\frac{g(a_ib_{i+1})+g(a_ib_i)}{\sigma(a_i)}=1.
\end{align*}
Here, the third and fourth equalities holds as $f_P$ is an SNMF on $P$ and $g$ is a flow of $(H, \sigma)$. Similarly, one can show that if $(x,i)\in L_R(k+1)$, then
$$\sum_{(y,j):\ (x,i)\sim_G(y,j)}f((x,i)(y,j))=\frac{1}{N_P(k+1-i)}(g(a_{i-1}b_{i})+g(a_ib_i))=\frac{N_R(k)}{N_R(k+1)}.$$
\end{proof}
Thus all that remains to show is that $(H,\sigma)$ has a flow $g$. As $H$ is a tree (more precisely a path), if a flow exists, it is unique, and can be easily calculated. Indeed, the only possibility is the function $g$ defined as
$$g(a_{\ell}b_{\ell})=\sigma(b_0)-\sigma(a_1)+\sigma(b_1)-\dots+\sigma(b_{\ell})=\frac{N_R(k)}{N_R(k+1)}\sum_{i=0}^{\ell}N_P(k+1-i)-\sum_{i=0}^{\ell-1}N_P(k-i),$$
for every $\ell\in [t]$, and
$$g(a_{\ell}b_{\ell+1})=\sigma(a_0)-\sigma(b_0)+\sigma(a_1)-\dots+\sigma(a_{\ell})-\sigma(b_{\ell})=\sum_{i=0}^{\ell}N_P(k-i)-\frac{N_R(k)}{N_R(k+1)}\sum_{i=0}^{\ell}N_P(k+1-i),$$
for every $\ell\in [t-1]$. This function $g$ clearly satisfies that 
\begin{itemize}
    \item $g(a_0b_0)=\sigma(b_0)$,
    \item  $g(a_{\ell}b_{\ell})+g(a_{\ell}b_{\ell+1})=\sigma(a_{\ell})$ for $\ell\in [t-1]$,
    \item $g(a_{\ell}b_{\ell})+g(a_{\ell-1}b_{\ell})=\sigma(b_{\ell})$ for $\ell\in \{1,\dots,t-1\}$, and
    \item $g(a_{t-1}b_{t-1})=\frac{N_R(k)}{N_R(k+1)} \cdot N_R(k+1)-(N_R(k)-N_P(k-t+1))=\sigma(a_{t-1})$.
\end{itemize}
To show that $g$ is a flow of $(H, \sigma)$, it suffices therefore to show that $g$ is non-negative.
\begin{claim}
$g(e)\geq 0$ for every $e\in E(H)$.
\end{claim}

\begin{proof}
To simplify notation, for integers $i_1, i_2$ write $S_{i_1,i_2}=\sum_{i=i_1}^{i_2}N_P(i)$. Then $N_R(k)=S_{k-t+1,k}$ and $N_R(k+1)=S_{k-t+2,k+1}$. First, consider the edges $a_{\ell}b_{\ell}$. We have
\begin{align*}
N_R(k+1)g(a_{\ell}b_{\ell})&=S_{k-t+1,k}S_{k+1-\ell,k+1}-S_{k-t+2,k+1}S_{k+1-\ell,k}\\
&=S_{k+1-\ell,k}\cdot N_P(k-t+1)+N_P(k+1)\cdot S_{k-t+1,k-\ell},
\end{align*}
which is clearly positive. On the other hand, we have
\begin{align}\label{equ:long}
\begin{split}
N_R(k+1)g(a_{\ell}b_{\ell+1})=&S_{k-t+2,k+1}S_{k-\ell,k}-S_{k-t+1,k}S_{k+1-\ell,k+1}\\
=&N_P(k-\ell)\cdot S_{k-t+2,k}-N_P(k+1)\cdot S_{k-t+2,k-\ell-1}\\
&-N_P(k-t+1)\cdot S_{k+1-\ell,k+1}\\
=&\sum_{i=k-t+2}^{k-\ell-1}[N_P(k-\ell)N_P(i+\ell+1)-N_P(k+1)N_P(i)]\\
&+\sum_{i=k+1-\ell}^{k+1}[N_P(k-\ell)N_P(i+\ell-t+1)-N_P(k-t+1)N_P(i)]
\end{split}
\end{align}
Using the log-concavity of the rank sequence of $P$, every term in each of the two sums in~\eqref{equ:long} is non-negative, whence $g(a_{\ell}b_{\ell+1})$ must also be non-negative.
\end{proof}
We conclude that $g$ is a flow on $(H, \sigma)$, and thus that the function $f$ we defined from $g$ does indeed give rise to an SNMF of $R=P\times [t]$. We record below some properties of $f$ that we shall need later.
\begin{lemma}\label{lemma:small_weights}
Let $(x,\ell)\in L_R(k)$.
\begin{itemize}
    \item[1.] If $x\precdot_{P} y$, then $f((x,\ell)(y,\ell))\leq f_P(xy)$.
    \item[2.] Let $\Lambda$ be the maximum of $N_P(z_1)N_P(z_2)-N_P(z_3)N_P(z_4)$ among all integers $z_1, z_2, z_3, z_4$ satisfying $k-t+1\leq z_1, z_2, z_3, z_4\leq k+1$ and $z_1+z_2=z_3+z_4$, and let $N_{min}$ be the minimum of $N_P(k-t+1),\dots,N_P(k+1)$. Then
    $$f((x,\ell)(x,\ell+1))\leq \frac{\Lambda}{(N_{min})^2}.$$
\end{itemize}
\end{lemma}

\begin{proof}
We have already established 1.\ earlier (just above Claim~\ref{claim: f is an SNMF}), so consider 2. Recall that
$$f((x,\ell)(x,\ell+1))=\frac{1}{N_P(k-\ell)}\cdot g(a_{\ell},b_{\ell+1}).$$
By (\ref{equ:long}), $N_R(k+1)\cdot g(a_{\ell}b_{\ell+1})$ can be written as a sum of $t-1$ non-negative terms, each of which is at most $\Lambda$, so $g(a_{\ell}b_{\ell+1})\leq \frac{t\Lambda}{N_R(k+1)}$. But 
$$N_R(k+1)=\sum_{i=k+2-t}^{k+1}N_P(i)\geq t\cdot N_{min},$$
so
$$f((x,\ell)(x,\ell+1))=\frac{g(a_{\ell}b_{\ell})}{N_P(k-\ell)}\leq \frac{\Lambda}{(N_{\min})^2}.$$
\end{proof}

In the end, given $P$ and $t$, the SNMF $f$ we defined on $P\times [t]$ only depends on the SNMF $f_P$. Therefore, we will use the notation $f=f_P^{\times t}$ to denote this unique SNMF.

\subsection{Normalized flows in the hypergrid with small edge weights}\label{sect:normalized_flows2}

Now let us turn our attention to $P=[t]^n$. We define an SNMF of $[t]^n$ recursively using the product structure of the hypergrid and the method from the previous subsection. We let $f_1$ denote the (unique) SNMF of $[t]^1$ in which each edge receives weight $1$. Now having already defined an SNMF $f_m$ of $[t]^m$, we write $[t]^{m+1}=[t]^{m}\times [t]$, and set $f_{m+1}=f_m^{\times t}$. Write $f=f_n$.

If $xy$ is an edge of the cover graph of $[t]^{n}$ with $x\precdot y$, then there is a unique coordinate $m$ where $x$ and $y$ differ, and we can use the pair $(x;m)$ to represent the edge $xy$. Also, given a vertex $x\in [t]^n$, we write $x_{[m]}$ for the element of $[t]^{m}$ formed by the first $m$ coordinates of $x$, i.e. $x_{[m]}=(x(1),\dots,x(m))$. Furthermore, we say that a vertex $x\in [t]^m$ is \emph{normal} if 
$$\left||x|-\frac{(t-1)m}{2}\right|\leq 2tm^{3/5},$$
and say that the edge $(x;m)$ in $[t]^n$ is \emph{normal} if $x_{[m]}$ is normal in $[t]^m$.

\begin{lemma}\label{lemma:bounding_flow}
Let $(x;m)$ be a normal edge. Then $$f((x;m))=O\left(\frac{1}{m}\right).$$
\end{lemma}

\begin{proof}
    By Lemma \ref{lemma:small_weights} part 1. we have $f((x;m))\leq f_m((x_{[m]};m))$. Writing $x'=x_{[m]}$, it is enough to show that $f_m((x';m))=O(1/m).$ 
    
    Let $P=[t]^{m-1}$, and write $[t]^m=P\times [t]$. Let $k=|x'|$, and let $\Lambda$ be the maximum of $N_P(z_1)N_P(z_2)-N_P(z_3)N_P(z_4)$ among all $k-t+1\leq z_1,z_2,z_3,z_4\leq k+1$ satisfying $z_1+z_2=z_3+z_4$. Also, let $N_{min}$ the minimum of $N_P(k+1-t),\dots,N(k+1)$. Since $(x';m)$ is a right edge in $P\times [t]$, we have by Lemma~\ref{lemma:small_weights} part 2.\  that $f_m((x';m))\leq \frac{\Lambda}{(N_{min})^2}$. Lemma \ref{lemma:a+b=c+d bound} then implies that $\frac{\Lambda}{(N_{min})^2} = O\left(\frac{1}{m}\right)$, so the result follows. 
\end{proof}

Unfortunately, for small $m$ Lemma~\ref{lemma:bounding_flow} only provides a weak bound on the weight $f((x;m))$ of the edge $(x;m)$. To overcome this, we define a new SNMF by exploiting the symmetries of $[t]^n$. Indeed, by taking all permutations of the coordinates, and averaging, we get an SNMF $f^{*}$ which assigns small weights to every edge. We prove this in the next lemma, for which we shall need the following simple probabilistic claim. Let $S_n$ denote the set of permutations of $\{1,\dots,n\}$.
\begin{claim}\label{claim:random_permutation}
	Let $a_1,\dots,a_n\in [0,1]$, and let $\bm{\pi}$ be a random permutation in $S_n$ chosen from the uniform distribution. Let $1\leq m\leq n$, and let $\mathbf{X}=\sum_{i=1}^{m} a_{\mathbf{\bm{\pi}}(i)}$. Then for every $z>0$,
	$$\mathbb{P}(|\mathbf{X}-\mathbb{E}(\mathbf{X})|\geq z\sqrt{m})\leq 2\cdot e^{-\Omega(z^2)}.$$
\end{claim}
\begin{proof}
By Serfling's bound~\cite[Corollary 1.1]{Serfling74} applied to the sequences $a_1, \ldots, a_n$ and $1-a_1, \ldots, 1-a_n$, we have
\begin{align*}
\mathbb{P}(|\mathbf{X}-\mathbb{E}(\mathbf{X})|\geq z\sqrt{m})\leq 2\cdot e^{-\frac{2z^2}{1-(m-1)/n}}.
\end{align*}
\end{proof}

Say that $x\in [t]^n$ is \emph{good} if
$$\left||x|-\frac{(t-1)n}{2}\right|\leq tn^{3/5}.$$

\begin{lemma}\label{lemma:small_SNMF}
There exists an SNMF $f^*$ of $[t]^n$ such that for every edge $xy$ of the cover graph, where $x\precdot y$ and $x$ is good, we have 
$$f^{*}(xy)=O\left(\frac{\log n}{n}\right).$$
\end{lemma}

\begin{proof}
 For $\pi\in S_n$ and $x\in [t]^n$, let $\pi(x)\in [t]^n$ be defined by $\pi(x)(i)=x(\pi(i))$ for $i\in \{1,\dots,n\}$. Also, define $f_{\pi}$ by setting $f_{\pi}(xy)=f(\pi(x)\pi(y))$  for every edge $xy$. Clearly, $f_{\pi}$ is also an SNMF. As the convex combination of SNMF's is also an SNMF, the function $f^{*}$ defined by 
$$f^{*}:=\frac{1}{n!}\sum_{\pi\in S_n} f_{\pi}=\mathbb{E}(f_{\bm{\pi}})$$
is also an SNMF, where the expectation is taken over a random permutation $\bm{\pi}$ chosen from the uniform distribution on $S_n$. Let us show that $f^{*}$ has the required properties.

Let $xy$ be an edge of $[t]^n$ with $x\precdot y$  and $x$ good, and let $m$ be the index of the coordinate that $x$ and $y$ differ in. Let $\bm{\pi}$ be a random permutation, and let $\textbf{m}=\bm{\pi}(m)$ and $\mathbf{x}=\bm{\pi}(x)_{[\textbf{m}]}$. Then by Lemma~\ref{lemma:bounding_flow} and the fact that $||f||_{\infty}\leq 1$, we have 
\begin{equation}\label{equ:f_pi}
f_{\bm{\pi}}(xy)=f((\bm{\pi}(x);\textbf{m}))\leq \begin{cases}O(1/\textbf{m}) &\mbox{if }\mathbf{x}\mbox{ is normal in }[t]^m,
\\1 &\mbox{otherwise.}\end{cases}
\end{equation}
For every $m_0\in \{1,\dots,n\}$, the goodness of $x$ implies 
$$\mu:=\mathbb{E}(|\mathbf{x}|\ \big|\ m_0=\mathbf{m})=\frac{m_0}{n}|x|\in \left[\frac{(t-1)m_0}{2}-\frac{m_0t}{n^{2/5}},\frac{(t-1)m_0}{2}+\frac{m_0t}{n^{2/5}}\right].$$
Hence,
$$\mathbb{P}(\mathbf{x}\mbox{ is not normal}\ \big|\ m_0=\mathbf{m})\leq \mathbb{P}\left(||\mathbf{x}|-\mu|\geq m_0^{3/5}t\ \big|\ m_0=\mathbf{m}\right).$$
By Claim \ref{claim:random_permutation}, the right hand side is at most $2\cdot e^{-\Omega(m_0^{1/5})}=O(1/m_0)$. Combining this with (\ref{equ:f_pi}), we get
$$\mathbb{E}(f_{\bm{\pi}}(xy)\ |\ m_0=\mathbf{m})\leq \mathbb{P}(\mathbf{x}\mbox{ is not normal}\ \big|\ m_0=\mathbf{m})+O\left(\frac{1}{m_0}\right)=O\left(\frac{1}{m_0}\right).$$
From this we can conclude that
$$f^{*}(xy)=\sum_{m_0=1}^{n}\mathbb{P}(\mathbf{m}=m_0)\cdot \mathbb{E}(f_{\bm{\pi}}(xy)\ |\ m_0=\mathbf{m})=\frac{1}{n}\cdot O\left(\sum_{m_0=1}^{n}\frac{1}{m_0}\right)=O\left(\frac{\log n}{n}\right).$$
\end{proof}
\section{Regular chain covers and supersaturation}
\subsection{Regular chain covers}\label{sect:chain_covers}

Recall that a regular chain cover of a graded poset $P$ is a probability distribution $\phi$ on the set of maximal chains $\mathcal{C}$ such that for every $x\in L(i)$, we have  
$$\phi(C\in \mathcal{C}:x\in C)=\frac{1}{N(i)}.$$
Given an SNMF $f$ on $P$, one can define a regular chain cover as follows. We may assume that $P$ has a unique minimum, denoted by $0$, and unique maximum, denoted by $1$. Indeed, if $P$ has no unique minimum, add a new element $0$ which is covered by every element of $L(0)$, and set $f(0x)=1/N(0)$ for every $x\in L(0)$. Proceed similarly if $P$ has no unique maximum, adding a new element $1$ covering every element of the maximum level, and setting $f(x1)=1$ for every $x$ in the maximum level. We say that a chain $C$ of $P$ is \emph{skipless} if  $\{i: L(i)\cap C\neq \emptyset\}$ is an interval. Given a skipless chain $C=\{x_0,\dots,x_{k}\}$, we define
$$\phi(C)=\prod_{i=0}^{k-1}f(x_ix_{i+1}),$$
noting that if $C$ contains only one element, then $\phi(C)=1$. Given $x,y\in P$ such that $x\prec y$, write $[x,y]$ for the set of all skipless chains with minimal element $x$ and maximal element $y$, and $\phi([x,y])$ for the sum of $\phi(C)$ over all chains $C\in[x,y]$.

\begin{lemma}\label{lemma:chain_weights}
   Let $x\in L(i)$. Then
    \begin{itemize}
        \item[1.] $\phi([0,x])=\frac{1}{N(i)}$, and
        \item[2.] $\phi([x,1])=1$.
    \end{itemize}
\end{lemma}

\begin{proof}
    1. We prove this by induction on $i$. If $i=0$, the statement is trivial, so assume that $i\geq 1$. Letting $x_1,\dots,x_r$ denote the elements of $L(i-1)$ such that $x_j\precdot x$, we have
    $$\phi([0,x])=\sum_{j=1}^{r}\phi([0,x_j])f(x_jx)=\frac{1}{N(i-1)}\sum_{j=1}^{r}f(x_jx)=\frac{1}{N(i)}.$$
    Here, the last equality holds as $f$ is an SNMF.

    2. Let $L(n-1)$ be the highest level of $P$. We prove the statement by induction on $i$. If $i=n-1$, the statement is trivial, so assume that $i<n-1$. Letting $x_1,\dots,x_r$ denote the elements of $L(i+1)$ such that $x\precdot x_j$, we have $$\phi([x,1])=\sum_{j=1}^{r}\phi([x_j,1])f(x_jx)=\sum_{j=1}^{r}f(x_jx)=1.$$
    Here again, the last equality holds as $f$ is an SNMF.
\end{proof}

By the previous lemma, we have $\phi(\mathcal{C})=\phi([0,1])=1$, so $\phi$ gives rise to a probability distribution on $\mathcal{C}$. Also, for every $x\in L(i)$, we have
$$\phi(C\in\mathcal{C}:x\in C)=\phi([0,x])\cdot \phi([x,1])=\frac{1}{N(i)},$$
so $\phi$ is a regular cover. We refer to $\phi$ as \emph{the regular cover induced by $f$}.

\subsection{Supersaturation}\label{sect:supersaturation}

Now let us turn our attention to $P=[t]^{n}$. Given a subset $A\subseteq [t]^{n}$, let $\Delta(A)$ denote the maximum degree of the comparability subgraph of $[t]^{n}$ induced by $A$.  First, we prove that if $A$ is much larger than $\alpha(t,n)$, then $\Delta(A)$ is large. To this end, we begin with the case $n=2$.

\begin{lemma}\label{lemma:rectangle_saturation}
    Let $A\subseteq [t]^2$  be such that $|A|> t$. Then $\Delta(A)= \Omega((|A|/t)^2)$.
\end{lemma}

\begin{proof}
    If $|A|>t=\alpha(t,2)$, then $\Delta(A)\geq 1$, so the statement is true for $|A|\leq 16t$ by choosing the constant hidden by the $\Omega(.)$ notation sufficiently small. In what follows, assume that $|A|\geq 16t$. Let $k=\lfloor |A|/(16t)\rfloor\geq |A|/(32t)$ and $s=\lceil t/k\rceil\leq 2t/k\leq 64t^2/|A|$. For $(i,j)\in [s]\times [s]$, let $$T_{i,j}=\{(a,b)\in [t]^2: ki\leq a<(k+1)i, kj\leq b <(k+1)j\}.$$
    Then the sets $T_{i,j}$ form a partition of $[t]\times [t]$, and $|T_{i,j}|\leq k^2$. For $u,v\in [s]\times [s]$, write $u\sim v$ if $v-u$ is an integer multiple of $(1,1)$. Then $\sim$ is an equivalence relation with $2s-1$ equivalence classes. See Figure \ref{fig:grid} for an illustration. By the pigeonhole principle, there exists a subset $B\subset A$ of size at least $$\frac{|A|}{2s}\geq \frac{|A|^2}{128t^2}>k^2$$ and an equivalence class $E$ such that each element of $B$ belongs to some $T_u$, $u\in E$.  Note that as $|B|>k^2$, not every element of $B$ is contained in the same $T_u$. This means that there is some $x\in B$ with $x\in T_v$ and $|B\setminus T_v|\geq k^2/2$. Since every element of $B\setminus T_v$ is comparable to $x$, this implies $x$ has degree at least $k^2/2=\Omega((|A|/t)^2)$. 
\end{proof}

\begin{figure}
    \centering
    \begin{tikzpicture}
        \fill[black!10!white] (-0.25,-0.25) rectangle (1.25,1.25);
        \fill[black!10!white] (1.25,1.25) rectangle (2.75,2.75);
        \fill[black!10!white] (2.75,2.75) rectangle (4.25,4.25);
        
        \def\cellsize{0.5cm}
        \draw[step=\cellsize, dotted] (0, 0) grid (8*\cellsize, 8*\cellsize);
        \foreach \x in {0,0.5,...,4} {
           \foreach \y in {0,0.5,...,4} {
             \fill (\x, \y) circle (1pt);
           }
        }
        \node at (-0.3,-0.3) {$(0,0)$};
        \node at (4.3,4.3) {$(t-1,t-1)$};

        \draw[color=red] (-0.25,1.25) -- (4.25,1.25);
        \draw[color=red] (-0.25,2.75) -- (4.25,2.75);
        \draw[color=red] (1.25,-0.25) -- (1.25,4.25);
        \draw[color=red] (2.75,-0.25) -- (2.75,4.25);

    \end{tikzpicture}
    \caption{An illustration for the proof of Lemma \ref{lemma:rectangle_saturation}. The sets $T_{i,j}$ are separated by the red lines, and the gray cells denote an equivalence class of $\sim$.}
    \label{fig:grid}
\end{figure}
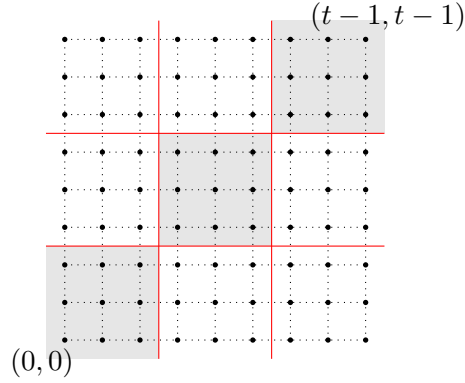
\noindent We now turn to the case of general $n$.
\begin{lemma}\label{lemma:weak_saturation}
For every $t,n\geq 2$, if $A\subseteq [t]^n$ is such that $|A|>\alpha(t,n)$, then 
$$\Delta(A)\geq \Omega\left(\left(\frac{|A|}{\alpha(t,n)}\right)^2\right).$$
\end{lemma}
\begin{proof}
Let $C>0$ be a constant specified later. As $\Delta(A)\geq 1$ by the condition $|A|>\alpha(t,n)$, the statement trivially holds for $|A|\leq C\alpha(t,n)$, assuming the constant hidden by the $\Omega(.)$ notation is sufficiently small. So we may assume that $|A|\geq C\alpha(t,n)$. By Lemma~\ref{lemma:rectangle_partition}, there exists an absolute constant $C_1>0$ and a partition of $[t]^n$ into rectangles $R_1,\dots,R_s$ with $R_i$ isomorphic to $[c_i]\times [c_i']$ for some $c_i,c_i'\leq C_1t\sqrt{n}$ and $s\leq C_1 t^{n-2}/n$. Write $u=C_1t\sqrt{n}$. By the pigeonhole principle, there exists some $i$ such that $|A\cap R_i|\geq |A|/s\geq |A|n/(C_1t^{n-2})$. If $|A\cap R_i|>u$, we can apply Lemma~\ref{lemma:rectangle_saturation} to conclude that 
 $$\Delta(A)\geq \Delta(A\cap R_i)=\Omega\left(\left(\frac{|A\cap R_i|}{u}\right)^2\right)=\Omega\left(\left(\frac{|A|}{t^{n-1}}\sqrt{n}\right)^2\right)=\Omega\left(\left(\frac{|A|}{\alpha(t,n)}\right)^2\right),$$
 where we used Lemma \ref{lemma:middle_level} in the last equality. It remains to show that $|A\cap R_i|>u$ holds. Note that if $|A|\geq C\alpha(t,n)$, then, again by Lemma~\ref{lemma:middle_level},
 $$|A\cap R_i|\geq \frac{|A|n}{C_1t^{n-2}}\geq \frac{C\alpha(t,n)n}{C_1t^{n-2}}\geq\frac{Ct\sqrt{n}}{C_1}.$$
 Hence, if $C>C_1^2$, we indeed have $|A\cap R_i|>u$.
\end{proof}

Next, we present stronger bounds for $\Delta(A)$ for those sets $A\subseteq [t]^n$ whose size is close to $\alpha(t,n)$. Let $f^{*}$ be the SNMF of $[t]^{n}$ we constructed in Lemma~\ref{lemma:small_SNMF}, and let \begin{align}\label{eq:def_of_T}
T:=\bigcup_{|i-(t-1)n/2|\leq tn^{3/5}} L(i)\end{align} be the collection of all good vertices of $[t]^n$. Then by Lemma~\ref{lemma:small_weights}, the maximum $W$ of $f^{*}$ over all edges $xy$ with $x\precdot y$ and $x\in T$ satisfies $W=O(\frac{\log n}{n})$. Let $\phi$ be the regular cover induced by $f^{*}$, and let $\mathbf{C}$ be a random element of $\mathcal{C}$ chosen according to the probability distribution given by $\phi$. 
\begin{lemma}\label{lemma:xy_in_C}
 Let $x\in L(i)$ for some $i$: $\vert i-(t-1)n/2\vert \leq tn^{3/5}$, and let $y\in T$ such that $x\prec y$. If $k$ is a positive integer such that $|y-x|\geq k$, then
$$\mathbb{P}(x,y\in \mathbf{C})\leq \frac{k!\cdot W^{k}}{N(i)}.$$
\end{lemma}

\begin{proof}
We have 
$$\mathbb{P}(x,y\in \mathbf{C})=\phi([0,x])\cdot \phi([x,y])\cdot \phi([y,1])=\frac{\phi([x,y])}{N(i)}$$
by Lemma \ref{lemma:chain_weights}. Therefore, it is enough to prove that $\phi([x,y])\leq k!\cdot W^k$. 

We prove this by induction on $|y-x|$. First, consider the base case $|y-x|=k$. The number of skipless chains with endpoints $x$ and $y$ is at most $k!$, and for each such chain $D$, we have $\phi(D)\leq W^{k}$. Therefore, the statement follows. Now assume that $|y-x|>k$. Let $x_1,\dots,x_r\in L(i+1)$ be all the elements such that $x\precdot x_j\prec y$. Then
$$\phi([x,y])=\sum_{j=1}^{r}f^*(xx_j)\cdot \phi([x_j,y])\leq k!\cdot W^k\cdot \sum_{j=1}^rf^{*}(xx_j)\leq k!\cdot W^k.$$
Here, the first inequality holds by our induction hypothesis, while the last inequality holds as $f^{*}$ is an SNMF. This finishes the proof.
\end{proof}

\noindent Define the \emph{weight} of a set $A\subseteq [t]^n$ as 
$$w(A)=\sum_{i=0}^{(t-1)n}\frac{|A\cap L(i)|}{N(i)}.$$
We point out that $w(A)$ equals to the expectation of $|\mathbf{C}\cap A|$.

\begin{lemma}\label{lemma:strong_saturation}
Let $k$ be a positive integer, let $\delta>0$, and let $A\subseteq T$ such that $w(A)\geq k+\delta$. Then
$$\Delta(A)\geq \frac{\delta W^{-k}}{k!(k+\delta)}.$$
\end{lemma}

\begin{proof}
Let $G$ be the graph on the vertex set $A$ in which $x$ and $y$ are joined by an edge if $x\prec y$ and $|y-x|\geq k$. Let $\Delta$ be the maximum degree of $G$. Clearly, $G$ is a subgraph of the comparability graph induced by $A$, so we have $\Delta(A)\geq \Delta$.

Let $\mathbf{X}$ be the number of pairs $(x,y)$ such that $x,y\in \mathbf{C}\cap A$, $x\prec y$, and $|y-x|\geq k$. On one hand, we have $\mathbf{X}\geq |\mathbf{C}\cap A|-k$. Indeed, if the elements of $\mathbf{C}$ are $x_0\prec\dots \prec x_{r-1}$ with $r\geq k$, then $(x_0,x_k),\dots,(x_0,x_{r-1})$ are all such pairs. Therefore,
$$\mathbb{E}(\mathbf{X})\geq \mathbb{E}(|\mathbf{C}\cap A|-k)=w(A)-k\geq \frac{\delta}{k+\delta}w(A),$$
where the last inequality holds by our assumption $w(A)\geq k+\delta$. On the other hand, we have
$$\mathbb{E}(\mathbf{X})=\sum_{xy\in E(G)}\mathbb{P}(x,y\in \mathbf{C}).$$
If $xy\in E(G)$ with $x\prec y$ and $x\in L(i)\subseteq T$, we can use Lemma \ref{lemma:xy_in_C} to write
$$\mathbb{P}(x,y\in \mathbf{C})\leq \frac{k!W^k}{N(i)}.$$ As each vertex of $A$ sends out at most $\Delta$ edges, we get
$$\sum_{xy\in E(G)}\mathbb{P}(x,y\in \mathbf{C})\leq \sum_{i=0}^{(t-1)n}\Delta\cdot |L(i)\cap A|\cdot \frac{k! W^{k}}{N(i)}=\Delta\cdot w(A)\cdot k!\cdot W^{k}.$$
In conclusion,
$$\Delta\cdot w(A)\cdot k!\cdot W^{k}\geq \mathbb{E}(\mathbf{X})\geq \frac{\delta}{k+\delta}w(A).$$
Comparing the left- and right-hand sides, we get $\Delta\geq \frac{\delta W^{-k}}{k!(k+\delta)}.$
\end{proof}

\section{Containers}\label{sect:containers}

In this section, we put everything together to prove our main theorem.  Let $G$ be the comparability graph of $[t]^{n}$. In order to count independent sets in $G$, we employ the celebrated graph container method originating in the work of Kleitman and Winston\cite{KW82} and Sapozhenko~\cite{saphozenko05}. Let $T \subseteq [t]^n$ be the union of good levels of $[t]^n$ defined in~\eqref{eq:def_of_T}.
\begin{lemma}\label{lemma:container}
There exists $\mathcal{S}\subset 2^{T}$ and an assignment $\psi: \mathcal{S}\rightarrow 2^{T}$ such that 
\begin{itemize}
    \item[1.] $|S|\leq O(\alpha(t,n)\cdot \frac{(\log n)^2}{n})$ for every $S\in \mathcal{S}$,
    \item[2.] $|\psi(S)|\leq (1+\frac{1}{n})\cdot \alpha(t,n)$ for every $S\in \mathcal{S}$,
    \item[3.] for every independent set $I$ in $G[T]$, there exists $S\in \mathcal{S}$ such that $S\subseteq I\subseteq S\cup \psi(S)$.
\end{itemize}
\end{lemma}

\begin{proof}
We describe an algorithm, whose input is an independent set $I$, and whose output is a set $S$ and an assignment $\psi(S)$ such that properties 1.--3. are satisfied. Let $<$ be an arbitrary total ordering on the elements of $T$, and for $v\in T$, let $N(v)$ denote the set of neighbours of $v$ in $G[T]$.

\medskip
\hrule
\medskip
\noindent
\emph{The algorithm.}

Let $I$ be an independent set of $G[T]$, and let $A_0:=T$ and $S_0:=\emptyset$. For $i=1,2,\dots,$ the $i$-th step of the algorithm proceeds as follows. Consider the maximum degree vertices of $G[A_{i-1}]$, and let $v$ be the smallest among them with respect to the ordering $<$. Separate two cases.
\begin{description}
     \item[Case 1.]  $v\in I$. 
    
    Set $S_{i}:=S_{i-1}\cup \{v\}$ and $A_{i}:=A_{i-1}\setminus (N(v)\cup \{v\})$. In this case, call $i$ an \emph{increment} step.
    \item[Case 2.] $v\not\in I$.
    
    Set $S_{i}:=S_{i-1}$ and $A_{i}:=A_{i-1}\setminus \{v\}$.
\end{description}

If $|A_i|\leq (1+1/n)\cdot \alpha(t,n)$, then STOP the algorithm and OUTPUT $S:=S_i$ and $\psi(S):=A_i$. Otherwise, move to step $i+1$.
\medskip
\hrule
\medskip

Let $\mathcal{S}$ be the family of all the sets $S$ the algorithm might produce for some independent set $I$. It is a standard argument to show that $\psi(S)$ is well-defined, that is, if two inputs of the algorithm produce the same $S$, then they also produce the same $\psi(S)$.  
It is easy to see that the family $\mathcal{S}$ and the assignment $\psi:\mathcal{S}\rightarrow 2^{T}$ thus defined satisfy properties 2.\ and 3., so all that remains is to prove 1.\ holds as well.

In order to bound the size of $S\in \mathcal{S}$, we count the number of \emph{increment steps}. To do this, we divide the steps of the algorithm into \emph{three phases} according to the size of $A_i$ and the supersaturation result we apply.

\begin{description}
    \item[Phase 1.] $|A_i|\geq n\cdot \alpha(t,n)$.

    Divide this phase into further subphases. For an integer $\ell$ such that $\lfloor \log_2 n\rfloor<\ell<\lfloor\log_2 (|T|/\alpha(t,n))\rfloor$, we let subphase $\ell$ consist of those steps $i$ for which $\alpha(t,n)\cdot 2^{\ell}\leq |A_i|\leq \alpha(t,n)\cdot 2^{\ell+1}$. By Lemma \ref{lemma:weak_saturation}, we have that for such $i$ the maximum degree of $A_i$ satsifies
    $$\Delta(A_i)\geq \Omega\left(\left(\frac{|A_i|}{\alpha(t,n)}\right)^2\right)=\Omega\left(2^{2\ell}\right).$$
    Hence, if step $i+1$ is an increment step in subphase $\ell$, then  $|A_{i+1}|\leq |A_{i}|-\Omega(2^{2\ell})$. This shows that subphase $\ell$ contains at most $O(\alpha(t,n)/2^{\ell})$ increment steps.

    We conclude that  phase 1. includes at most 
    $$O(\alpha(t,n))\cdot \sum_{\ell\geq \lfloor \log_2 n\rfloor}2^{-\ell}=O\left(\frac{\alpha(t,n)}{n}\right)$$ increment steps in total.

    \item[Phase 2.] $n\cdot \alpha(t,n)>|A_i|\geq 3\cdot\alpha(t,n)$

    Apply Lemma \ref{lemma:strong_saturation} with $k=2$ and $\delta=1$. Note that if $|A_i|\geq 3\alpha(t,n)$, then $w(A)\geq 3$, so we get that $\Delta(A_i)\geq W^{-2}/6=\Omega((n/\log n)^2)$. Therefore, if step $i+1$ is an increment step, then  $|A_{i+1}|\leq |A_{i}|-\Omega((n/\log n)^2)$. Hence, there are at most $O(\alpha(t,n)(\log n)^2/n)$ increment steps in phase 2.

    \item[Phase 3.] $3\cdot \alpha(t,n)>|A_i|\geq (1+\frac{1}{n})\cdot \alpha(t,n)$

Again, we divide this phase into further subphases. For $\ell=0,\dots,\lceil \log_2 (2n)\rceil$, let subphase $\ell$ consist of those steps $i$ for which 
$$\alpha(t,n)\cdot\left(1+\frac{2^{\ell}}{n}\right)\leq |A_i|\leq \alpha(t,n)\cdot \left(1+\frac{2^{\ell+1}}{n}\right).$$
Suppose we are in subphase $\ell$. Applying Lemma \ref{lemma:strong_saturation} with $k=1$ and $\delta=2^{\ell}/n$, we have $w(A_i)\geq 1+\delta$ and hence 
$$\Delta(A_i)\geq \frac{\delta}{1+\delta}\cdot W^{-1}\geq \Omega\left(\frac{2^{\ell}}{\log n}\right).$$
Therefore, if step $i+1$ is an increment step, then  $|A_{i+1}|\leq |A_{i}|-\Omega(2^{\ell}/\log n)$. Hence, there are at most $O(\alpha(t,n)(\log n)/n)$ increment steps in subphase $\ell$.

In conclusion, as there are $O(\log n)$ subphases, phase 3. includes at most $O(\alpha(t,n)(\log n)^2/n)$ increment steps in total.
\end{description}

Summarizing the three phases, for every input $I$ our algorithm goes through at most $O(\alpha(t,n)(\log n)^2/n)$ increment steps, giving us the requisite uniform upper bound on the size of the output set $S$. This concludes the proof.
\end{proof}

\begin{proof}[Proof of Theorem \ref{thm:main}]
    First, let us count the number of antichains contained inside the collection of good levels $T$, and let us denote their number by $A(T)$. Let $\mathcal{S}\subset 2^{T}$ and $\psi:\mathcal{S}\rightarrow 2^{T}$ be given by Lemma \ref{lemma:container}. Each element of $\mathcal{S}$ is an antichain of size $O(\alpha(t,n)\cdot \frac{(\log n)^2}{n})$, so we can apply Theorem \ref{thm:small_antichains} with $k=\Theta(n/(\log n)^2)$ to get 
    $$|\mathcal{S}|=\exp_2\left[O\left(\frac{(\log n)^3}{n}\cdot \alpha(t,n)\right)\right].$$
    Every antichain $A$ of $T$ is contained in $S\cup \psi(S)$ for some $S\in  \mathcal{S}$, and 
    $$|S\cup \psi(S)|\leq \left(1+O\left(\frac{(\log n)^2}{n}\right)\right)\cdot \alpha(t,n).$$ Thus, the number of antichains in $T$ is at most
    \begin{align*}
        A(T)&\leq \sum_{S\in \mathcal{S}}\exp_2(|S\cup \psi(S)|)\leq |\mathcal{S}|\cdot\exp_2\left[\left(1+O\left(\frac{(\log n)^2}{n}\right)\right)\cdot\alpha(t,n)\right]\\
        &=\exp_2\left[\left(1+O\left(\frac{(\log n)^3}{n}\right)\right)\cdot\alpha(t,n)\right]
    \end{align*}

    Now let us count the number of antichains in $U=[t]^{n}\setminus T$, and let us denote this number by $A(U)$. Writing $\ell=\left\lceil \frac{(t-1)n}{2}-n^{3/5}t-1\right\rceil$, we have
    $$U=\bigcup_{i=0}^{\ell}L(i)\cup \bigcup_{i=(t-1)n-\ell}^{(t-1)n}L(i).$$
    Thus, by the LYM-property (which holds in $[t]^n$ by Lemma~\ref{lemma:LYM}), the width of $U$ is $N(\ell)$. We can use Lemma \ref{lemma:level_upperbound} to bound $N(\ell)$ as 
    $$N(\ell)\leq t^{n-1}\cdot e^{-\Omega(n^{1/5})}=O\left(\frac{\alpha(t,n)}{n}\right).$$
    Now we can get an upper bound on $A(U)$ by counting all antichains in $[t]^n$ of size at most $O(\alpha(t,n)/n)$, for which we can use  Theorem \ref{thm:small_antichains} again with $k=\Theta(n)$. Thus, we get
    $$A(U)\leq \exp_2\left[O\left(\frac{\log n}{n}\cdot \alpha(t,n)\right)\right].$$

    Finally, observe that every antichain $A\subseteq [t]^n$ can be written as $A=A_T\cup A_U$, where $A_T$ is an antichain in $T$, and $A_U$ is an antichain in $U$. Therefore,
    $$A(t,n)\leq A(T)\cdot A(U)\leq \exp_2\left[\left(1+O\left(\frac{(\log n)^3}{n}\right)\right)\cdot \alpha(t,n)\right].$$
\end{proof}

\section{Concluding remarks}

We finish our paper by proving a lower bound on $A(t,n)$.
\begin{claim}
For every  $n$ sufficiently large and $t\geq 5$,
$$\log_2 A(t,n)\geq (1+2^{-2n-2})\cdot \alpha(t,n).$$
\end{claim}

\begin{proof}
Let $m=\lfloor (t-1)n/2\rfloor$, and consider the bipartite subgraph $G$ of the comparability graph induced by the levels $L(m)$ and $L(m+1)$. Using that $G$ has  maximum degree $n$, we can show that $G$ contains at least $\exp_2((1+2^{-2n-2})\cdot \alpha(t,n))$ independent sets.

Indeed, let $A\subseteq L(m+1)$ be arbitrary, and let $N(A)\subseteq L(m)$ denote the neighborhood of $A$ in $G$. Then $|N(A)|\leq n|A|$ (and in particular $\vert L(m)\vert \leq n \vert L(m+1)\vert$). Let $k=\lfloor N(m+1)\cdot 2^{-2n}\rfloor$, and note that by Lemma~\ref{lemma:middle_level} we have $k\geq 1$ if $t\geq 5$ and $n$ is sufficiently large. Consider sets of the form $A\cup B$, where $A\subseteq L(m+1)$, $|A|=k$, and $B \subseteq L(m)\setminus N(A)$. Every such set $A\cup B$ is an independent set, and the number of such sets is at least
\begin{align*}
    \binom{N(m+1)}{k}\cdot 2^{N(m)-kn}&\geq \exp_2\left(N(m)-kn+k\log_2 \frac{N(m+1)}{k}\right)\\
     &\geq \exp_2(N(m)+2^{-2n-1}n\cdot N(m+1))\\
     &\geq \exp_2(N(m)\cdot (1+2^{-2n-1})).
\end{align*}
\end{proof}
A similar bound  for $t\in \{2,3,4\}$ can be established as well by using sharper estimates on the binomial coefficients, however, we omit the proof. It might be reasonable to conjecture that this lower bound is closer to optimal than our upper bound in Theorem \ref{thm:main}, that is, $\log_2 A(t,n)=(1+2^{-\Theta(n)})\cdot \alpha(t,n)$. This is known to be true if $t=2$ by the result of Korshunov \cite{korshunov80}. However, such a result, if true, is well beyond the methods of our paper.

\subsection*{Acknowledgements}
Research on this paper was partially supported by Swedish Research Council grant VR 2021-03687 and postdoctoral grant 213-0204 from the Olle Engkvist Foundation.

\section*{Appendix: proof of Lemma~\ref{lemma:derivatives}}

Here, we present the proof of Lemma~\ref{lemma:derivatives}. Recall from Section~\ref{sect:levels} that we are interested in the large $n$ behaviour of the function $f$ defined in~\eqref{eq: f-def}, and write $o_n(1)$ for any function tending to $0$ as $n$ tends to infinity.  Let $ q\in\left[-tn^{2/3}, tn^{2/3}\right]$, $k=\frac{(t-1)n}{2}-q$, and let $p\geq 0$ be such that 
$$\mu:=\frac{1}{\alpha}\cdot \sum_{j=0}^{t-1}j\cdot p^j=\frac{k}{n},$$
where $\alpha:=\sum_{j=0}^{t-1}p^j$. Recall that $p=1+O(\frac{q}{t^2 n})=1+o_n(1)\cdot \frac{1}{t}$ by Lemma~\ref{lemma:p}, and thus $\alpha=(1+o_n(1))\cdot t$. Write $Y_1:=X_1-k/n$, where $X_1$ is the random variable defined in~\eqref{eq:def_of_X_i}; we thus have
$$\mathbb{P}(Y_1=j-\mu)=\alpha^{-1}\cdot p^j,\ \ \  \mbox{ for } j\in [t],$$ and $\mathbb{E}(Y_1)=0$. As noted in Section~\ref{sect:levels} the characteristic function of $Y_1$ is  
$$\varphi\left(y\right)=\mathbb{E}e^{iyY_1}= \frac{1}{\alpha}\sum_{j=0}^{t-1} p^je^{iy\left(j-\mu\right)},$$
and we defined a Hermitian function $f$ by
\[
f\left(x\right):=\frac{1}{2\pi}\int_{-\pi}^{\pi}e^{-ixy}\varphi\left(y\right)^{n}dy.
\]
With these definitions, the statement of Lemma~\ref{lemma:derivatives} is as follows, where $\Vert\cdot\Vert_{\infty}$ denotes the standard supremum norm.
\begin{lemma}\label{lemma:appendix}
We have  
\begin{enumerate}
\item $f\left(x\right)=\Theta\left(\frac{1}{t\sqrt{n}}\right)$ for $x\in\left[-t,t\right]$,
\item $\left\Vert f'\right\Vert _{\infty}=O\left(\frac{1}{t^{2}n}\right)$,
\item $\left\Vert f''\right\Vert _{\infty}=O\left(\frac{1}{t^{3}n^{3/2}}\right)$. 
\end{enumerate}
\end{lemma}
\begin{proof}
We lay the ground for a proof by providing a number of estimates on the value of $\phi(y)$ at various points in the interval $[0, \pi]$. 

\begin{claim} \label{claim:First bound} There exists $\delta > 0$ such that the following holds for every sufficiently large $n$. If $y \in \left[-\frac{2.26}{t},\frac{2.26}{t}\right]$, then $$|\varphi(y)| < 1 - \delta y^2t^2\leq e^{-\delta y^2 t^2}.$$
\end{claim}
\begin{proof}
The function $\varphi(y)$ can be written as 
$$\varphi(y) =\mathbb{E}e^{iyY_1}= \mathbb{E}\left[\sum_{m = 0}^{\infty} \frac{\left(iyY_1\right)^m}{m!}\right]= \sum_{m = 0}^{\infty} \frac{\mathbb{E}\left((Y_1)^m\right)}{m!} (iy)^m.$$
\noindent Let us write $a_{i} = \frac{i^m\mathbb{E}\left((Y_1)^m\right)}{ m!}$ for the coefficient of $y^m$ in the power series expression for $\varphi(y)$ thus obtained.  It is easy to check that  $a_0 = 1$ and $a_1 = 0$. Let us estimate $a_2$.

Using that $p=1+o_n(1)\cdot\frac{1}{t}$ and $\mu= \frac{t-1}{2}(1+o_n(1))$, we have
$$\mathbb{E}\left((Y_1)^2\right) = \alpha^{-1}\cdot \sum_{j = 0}^{t-1} (j-\mu)^2p^j=\frac{t^2-1}{12}(1+o_n(1)).$$
Thus $a_2=-\frac{t^2-1}{24}(1+o_n(1))$. By a similar argument, we get $a_3=\mathbb{E}((Y_1)^3) = o_n(1)\cdot t^3$. Hence we have $|a_3y^3| = o_n(1)\cdot |ty|^2$ provided that $|y| \leq \frac{3}{t}$.

Now let us focus on the the rest of the terms. By the choice of $k$, we have $\mu \in [0.999,1.001]\cdot \frac{t-1}{2}$ for $n$ sufficiently large. Thus, $Y_1\in [-1.001\cdot \frac{(t-1)}{2},1.001\cdot \frac{(t-1)}{2}]$. Writing $\beta_1 = \frac{1.001}{2}$, we thus have the trivial upper bound $\mathbb{E}\left(|(Y_1)^m|\right) \leq \left(1.001\cdot \frac{(t-1)}{2}\right)^m$. Thus for $m \geq 4$ we have 
$$\sum_{m = 4}^{\infty} |a_m y^m| \leq \frac{(\beta_1 y(t-1))^4}{24} \sum_{m = 0}^{\infty} \frac{|\beta_1 yt|^m}{(m+4)!/24}.$$ Since $(m+4)! \geq 24m!$ for every $m \geq 0$, it follows that $\sum_{m = 0}^{\infty} \frac{|\beta_1 yt|^m}{(m+4)!/24} \leq e^{|\beta_1 yt|}$, and thus we have 
$$\sum_{m = 4}^{\infty} |a_my^m| \leq \frac{(y(t-1))^2}{24} \cdot (yt)^2\beta_1^4 \cdot e^{|\beta_1 yt|}$$
 Since $|yt| \leq 2.26$, it follows that $(yt)^2 \beta_1^4 \cdot e^{|\beta_1 yt|} \leq 0.99$. Let $\delta_1 = \frac{1}{1000}$, then $$\sum_{m=4}^{\infty} \left|a_m y^m\right| < \frac{0.99}{24} y^2(t-1)^2 < \frac{1}{24}(1 - 3\delta_1) y^2(t^2-1).$$
For $n$ sufficiently large, we also have $1 - a_{2}y^2 \leq 1 - \frac{t^2-1}{24}(1 - \delta_1)y^2$ for every $y \in \left[-\frac{3}{t},\frac{3}{t}\right]$.
As we also have $|a_{3}y^3| \leq \frac{\delta_1}{24} y^2(t^2-1)$ if $n$ is sufficiently large, we can combine these bounds to get 
$$|\varphi(y)| \leq |1-a_2y^2|+|a_3 y^3|+\sum_{m=4}^{\infty}|a_m y^m|\leq 1 - \frac{\delta_1}{24}y^2(t^2-1)\leq 1 - \frac{\delta_1}{48}y^2t^2.$$ Thus we can set  $\delta = \frac{1}{48000}$, which completes the proof. 
\end {proof}

\begin{claim}\label{claim:pi^2 bound}
Let $p > 0$ be a constant chosen so that $\frac{8(1+p^2)}{\pi^2} < 2p$. Then for all $y \in \left[-\frac{\pi}{2},\frac{\pi}{2}\right]$ we have $1 + p^2 - 2p\cos(y) \geq \frac{4(1+p^2)}{\pi^2}y^2$
\end{claim}
\begin{proof} 
Consider the function $a(y) = 1+p^2 - 2p\cos(y) - \frac{4(1+p^2)}{\pi^2}y^2$. Note that $a$ is an even function whose first and second derivatives are given by $a'(y) = 2p\sin(y) - \frac{8(1+p^2)y}{\pi^2}$ and $a''(y) = 2p\cos(y) - \frac{8(1+p^2)}{\pi^2}$ respectively.
Thus there exists $\alpha \in \left[0,\frac{\pi}{2}\right]$ so that $a'(y) \geq 0$ for $y \in \left[0,\alpha\right]$ while $a'(y) \leq  0$ for $y \in \left[\alpha,\frac{\pi}{2}\right]$. Thus $a$ attains its minimum at $y = 0$ or $y = \frac{\pi}{2}$. Since $a(0) = (1-p)^2 \geq 0$ and $a\left(\frac{\pi}{2}\right) = 0$, it follows that $1 + p^2 - 2p\cos(y) \geq \frac{4(1+p^2)}{\pi^2}y^2$.
\end{proof}
 
\begin{claim} \label{claim:second bound}
For $|y| \in \left(0,\frac{\pi}{2}\right]$, we have $|\varphi(y)| \leq \frac{2.25}{|y|t}$, while for $|y|\in \left[\frac{\pi}{2},\pi\right]$, $|\varphi(y)|\leq \frac{1.5}{t}$ holds.
\end{claim}
\begin{proof}
Write $\varphi(y)$ as 
$$\varphi(y) = \frac{1}{\alpha} \cdot \frac{1-p^te^{ity}}{1-pe^{iy}} \cdot e^{-i\mu y}.$$
Here, we have $|\frac{1}{\alpha}|\leq \frac{1}{0.99 t}$ for $n$ sufficiently large. To bound the size of the second term, observe that for $n$ sufficiently large $4(1+p^2)=4(1+o_n(1))<\pi^2$, and thus that $|1-pe^{iy}| = \sqrt{1+p^2 - 2p\cos(y)} \geq \frac{2\sqrt{1+p^2}\cdot |y|}{\pi}$ by Claim \ref{claim:pi^2 bound}. Hence, 
$$\left|\frac{1-p^te^{ity}}{1-pe^{iy}}\right|\leq \frac{2}{|1-pe^{iy}|} \leq \frac{\pi}{\sqrt{(1+p^2)}\cdot|y|}.$$ Since for $n$ sufficiently large we have $0.999 < p < 1.001$, it follows that 
$$|\varphi(y)| \leq \frac{\pi}{\sqrt{1 + 0.999^2}\cdot |y| \cdot 0.99t} < \frac{2.25}{|y|t}.$$
For the second part, note that if $|y|\in \left[\frac{\pi}{2},\pi\right]$, then $|1-pe^{iy}|\geq |1-e^{iy}|-0.01\geq 1.4$. Hence,
$$|\phi(y)|\leq \frac{1}{0.99t}\cdot \frac{2}{1.4}\leq \frac{1.5}{t}.$$
\end{proof}

\begin{proof} [Proof of Lemma~\ref{lemma:appendix} part 2] 
 We have 
\[
f'\left(x\right)=\frac{1}{2\pi}\int_{-\pi}^{\pi}-iye^{-ixy}\varphi\left(y\right)^{n}dy,
\]
so 
$$||f'||_{\infty}\leq \int_{-\pi}^{\pi}|y|\cdot |\varphi\left(y\right)|^{n}dy.$$
Assume that $n$ is sufficiently large. Since $\vert y \phi(y)\vert =\vert (-y)\phi(-y)\vert$, it is enough to bound the contribution to this integral from the interval $y\in [0,\pi]$. Divide $[0,\pi]$ into three intervals, $I_1=[0,\frac{2.26}{t}]$, $I_2=[\frac{2.26}{t},\frac{\pi}{2}]$, and $I_3=[\frac{\pi}{2},\pi]$. By Claim \ref{claim:First bound}, there exists some $\delta>0$ such that
$$\int_{I_1}|y|\cdot |\varphi\left(y\right)|^{n}dy\leq \int_{I_1} y\cdot e^{-\delta ny^2t^2} dy = \frac{1}{2\delta nt^2}\left(1 - e^{-n \delta (2.26)^2}\right) = O\left(\frac{1}{nt^2}\right).$$
Further, by the first part of  Claim \ref{claim:second bound} we have
\begin{align*}
    \int_{I_2}|y|\cdot |\varphi\left(y\right)|^{n}dy&\leq \int_{I_2} y \cdot \left(\frac{2.25}{yt}\right)^n dy\\
    &=\left(\frac{2.25}{t}\right)^n\cdot \int_{I_2} y^{-(n-1)}dy\\
    &=\frac{2.25^2}{(n-2)t^2} \left(\left(\frac{2.25}{2.26}\right)^{n-2} - \left(\frac{2\cdot 2.25}{t\pi}\right)^{n-2}\right)\\
    &=O\left(\frac{1}{nt^2}\right).
\end{align*}
Finally, using the second part of Claim \ref{claim:second bound}, we get
$$\int_{I_3}|y|\cdot |\varphi\left(y\right)|^{n}dy\leq \int_{I_3} y\cdot \left(\frac{1.5}{t}\right)^n dy = O\left(\frac{1}{nt^2}\right).$$
This concludes the proof.
\end{proof}

\begin{proof}[Proof of Lemma~\ref{lemma:appendix} part 3]
The proof of this is almost the same as the proof of part 2. However, as there are some small nuances, let us present the complete proof. We have
\[
f''\left(x\right)=\frac{1}{2\pi}\int_{-\pi}^{\pi}y^2 e^{-ixy}\varphi\left(y\right)^{n}dy,
\]
so 
$$||f''||_{\infty}\leq \int_{-\pi}^{\pi}|y|^2\cdot |\varphi\left(y\right)|^{n}dy.$$
Again, we only consider the interval $[0,\pi]$, and we divide it into three parts $I_1,I_2,I_3$, in exactly the same way as before. 
By subdividing $I_1$ into $\lceil 2.25\sqrt{n}\rceil $ intervals of length at most $\frac{1}{t\sqrt{n}}$ and applying Claim~\ref{claim:First bound}, we get the estimate  
\begin{align*}
    \int_{I_1}|y|^2\cdot |\varphi\left(y\right)|^{n}dy&\leq \int_{I_1} y^2\cdot e^{-\delta ny^2t^2} dy \\
    &\leq \sum_{j=1}^{\lceil 2.25\sqrt{n}\rceil} \frac{1}{t\sqrt{n}} \left(\frac{j}{t\sqrt{n}}\right)^2 e^{-\delta  n \left((j-1)\cdot/t\sqrt{n}\right)^2 t^2}\\
    &=\frac{1}{t^3n^{3/2}}\sum_{j=1}^{\lceil 2.25\sqrt{n}\rceil} j^2e^{-\delta(j-1)^2}=O\left(\frac{1}{t^3n^{3/2}}\right).
\end{align*}
Next, using the first part of  Claim \ref{claim:second bound}, we have
\begin{align*}
    \int_{I_2}|y|^2\cdot |\varphi\left(y\right)|^{n}dy&\leq \int_{I_2} y^2 \cdot \left(\frac{2.25}{yt}\right)^n dy\\
    &=\left(\frac{2.25}{t}\right)^n\cdot \int_{I_2} y^{-(n-2)}dy\\
    &=\frac{2.25^2}{(n-3)t^3} \left(\left(\frac{2.25}{2.26}\right)^{n-3} -\cdot \left(\frac{2\cdot 2.25}{t\pi}\right)^{n-3}\right)\\
    &=O\left(\frac{1}{t^3n^{3/2}}\right).
\end{align*}
Finally, using the second part of Claim \ref{claim:second bound}, we get
$$\int_{I_3}|y|^2\cdot |\varphi\left(y\right)|^{n}dy\leq \int_{I_3} y^2\cdot \left(\frac{1.5}{t}\right)^n dy = O\left(\frac{1}{t^3 n^{3/2}}\right).$$
\end{proof}

\begin{proof}[Proof of Lemma~\ref{lemma:appendix} part 1]
Finally, let us prove 1. Let $Y_1,\dots,Y_n$ be i.i.d. copies of $Y_1$, and let $X=Y_1+\dots+Y_n$. Recall from Section~\ref{sect:levels} that for every integer $j$, $f(j)=\mathbb{P}(X=j)$. We have $\mathbb{E}(Y_1)=0$, $\sigma^2=\mathbb{E}((Y_1)^2)=\Theta(t^2)$ and $\rho=\mathbb{E}((Y_1)^3)\leq t^3$. Therefore, by the Berry--Esseen theorem, for every $x$,
$$\left|\mathbb{P}\left(X\leq x\sqrt{n}\sigma\right)-\Phi(x)\right|\leq O\left(\frac{\rho}{\sigma^3 \sqrt{n}}\right)\leq \frac{C}{\sqrt{n}}$$
for some absolute constant $C>0$, where $$\Phi(x)=\frac{1}{\sqrt{2\pi}}\int_{-\infty}^{x}e^{-u^2/2}du$$ is the cumulative distribution function of the standard normal distribution. Choose a constant $c>10$ such that if $x_1=-\frac{c}{\sqrt{n}}$ and $x_2=\frac{c}{\sqrt{n}}$, then $\Phi(x_2)-\Phi(x_1)\geq \frac{3C}{\sqrt{n}}$. Then
\begin{align*}
    \mathbb{P}\left(x_1\sqrt{n}\sigma \leq X\leq x_2\sqrt{n}\sigma\right)\geq \Phi(x_2)-\Phi(x_1)
    -\sum_{i=1}^2\left|\mathbb{P}(X\leq x_i\sqrt{n}\sigma)-\Phi(x_i)\right| \geq \frac{C}{\sqrt{n}}.
\end{align*}
Similarly, we can prove that 
$$\mathbb{P}(x_1\sqrt{n}\sigma \leq X\leq x_2\sqrt{n}\sigma)=O\left(\frac{1}{\sqrt{n}}\right).$$
As $X$ takes values in $\mathbb{Z}$, the lower bound implies that there exists some $x_0\in [x_1\sqrt{n}\sigma,x_2\sqrt{n}\sigma]=[-c\sigma,c\sigma]$ such that 
$$f(x_0)=\mathbb{P}(X=x_0)\geq \frac{C}{1+2c\sigma\sqrt{n}}=\Omega\left(\frac{1}{t\sqrt{n}}\right).$$
Now let $x\in [-t,t]\subseteq [-c\sigma,c\sigma]$. By Taylor's theorem, there exists some $\xi\in [-t,t]$ such that
$$\left|f(x) - f(x_0)\right| = \left| (x - x_0)f'(\xi)\right| = O\left(\frac{1}{tn}\right) = O\left(\frac{f(x_0)}{\sqrt{n}}\right),$$
where the second equality holds since $||f'||_{\infty}=O\left(\frac{1}{nt^2}\right)$, as proved in part 2. In conclusion, for every $x \in [-t,t]$ we have $f(x) = \Theta(f(x_0))=\Theta(\frac{1}{t\sqrt{n}})$, as desired.
\end{proof}

\end{proof}


\begin{thebibliography}{99}
\bibitem{anderson67}
I. Anderson.
\emph{On primitive sequences.}
J. London Math. Soc. s1-42 (1) (1967): 137--148.

\bibitem{anderson68}
I. Anderson.
\emph{On the divisors of a number.}
J. London Math. Soc. s1-43 (1) (1968): 410--418.

\bibitem{book_Anderson}
I. Anderson.
\emph{Combinatorics of finite sets.}
Courier Corporation, 2002.

\bibitem{andrews}
G. Andrews, and K. Ericson. 
\emph{Integer Partitions.}
Cambridge University Press, 2004. 

\bibitem{BTW}
J. Balogh, A. Treglown, and A. Zs. Wagner.
\emph{Applications of graph containers in the Boolean lattice.}
Random Structures Algorithms, 49(4) (2016): 845--872.

\bibitem{CCT09}
T. Carroll, J. N. Cooper, and P. Tetali.
\emph{Counting antichains and linear extensions in generalizations of the Boolean lattice.}
manuscript (2009)

\bibitem{dedekind}
R. Dedekind. 
\emph{\"Uber Zerlegungen von Zahlen durch ihre gr\"o{\ss}ten gemeinsamen Teiler.}
Festschrift Hoch. Braunschweig u. ges. Werke, II (1897): 103--148.

\bibitem{dilworth50}
R. Dilworth.
\emph{A Decomposition Theorem for Partially Ordered Sets.} 
Annals of Mathematics (1950): 161-166.


\bibitem{durrett}
R. Durrett.
\emph{Probability: Theory and Examples.}
Vol. 49. Cambridge university press, 2019.


\bibitem{ESz}
P. Erd\H{o}s, and G. Szekeres. 
\emph{A combinatorial problem in geometry.}
Compositio Math. 2 (1935): 463--470.

\bibitem{FPSS}
J. Fox, J. Pach, B. Sudakov, and A. Suk. 
\emph{Erd\H{o}s-Szekeres-type theorems for monotone paths and convex bodies.}
Proc. London Math. Soc. 105.5 (2012): 953--982.

\bibitem{harper74}
L. H. Harper.
\emph{The morphology of partially ordered sets.}
J. Comb. Theory A 17 (1974): 44--58.

\bibitem{kahn02}
J. Kahn. 
\emph{Entropy, independent sets and antichains: A new approach to Dedekind's problem.}
Proceedings of the American Mathematical Society 130(2) (2002): 371--378.

\bibitem{kleitman69}
D. J. Kleitman.
\emph{On Dedekind's problem: The number of monotone Boolean functions.}
manuscript (1969)

\bibitem{kleitman74}
D. J. Kleitman.
\emph{On an extremal property of antichains in partial orders.}
In Combinatorics (eds. M. Hall and J. H. van Lint), Math. Centre Tracts 55 (1974): 77--90. Mathematics Centre, Amsterdam.

\bibitem{KM75}
D. J. Kleitman, and G. Markowsky. 
\emph{On Dedekind's problem: the number of isotone Boolean functions.}
ii. Transactions of the American Mathematical Society, 213 (1975): 373--390.

\bibitem{KW82}
D. J. Kleitman, and K. J. Winston.
\emph{On the number of graphs without 4-cycles.}
Discrete Math., 6 (1982): 167--172.

\bibitem{korshunov80}
A. D. Korshunov. 
\emph{The number of monotone Boolean functions.}
(russian) Problemy Kibernet 38 (1980): 5--108.

\bibitem{macmahon}
P. A. MacMahon. 
\emph{Combinatory analysis.}
Two volumes (bound as one). Chelsea Publishing Co., New York, 1960.

\bibitem{MS12}
G. Moshkovitz, and A. Shapira. 
\emph{Ramsey theory, integer partitions and a new proof of the Erd\H{o}s-Szekeres theorem.}
Advances in Mathematics 262 (2012): 1107--1129.

\bibitem{NSS}
J. Noel, A. Scott, and B. Sudakov.
\emph{Supersaturation in posets and applications involving the container method.}
J. Combin. Theory, Ser. A 154 (2018): 247--284.

\bibitem{PST23}
J. Park, M. Sarantis, and P. Tetali.
\emph{Note on the number of antichains in generalizations of the Boolean lattice.}
preprint, arXiv:2305.16520 (2023)

\bibitem{pippenger}
N. Pippenger. 
\emph{Entropy and enumeration of Noolean functions.}
IEEE Trans. Inf. Theory, 45 (1999): 2096--2100.

\bibitem{PZ21}
C. Pohoata, and D. Zakharov. 
\emph{On the number of high-dimensional partitions.}
Preprint, arXiv:2112.02436 (2021)

\bibitem{saphozenko89}
A. Sapozhenko. 
\emph{The number of antichains in multilayered ranked sets.} (russian)
Diskret. Mat., 1(1) (1989): 110--128. 

\bibitem{saphozenko05}
A. Sapozhenko.
\emph{Systems of containers and enumeration problems.}
Stochastic Algorithms: Foundations and Applications, Springer Berlin Heidelberg (2005): 1--13.

\bibitem{Serfling74}
R. Serfling.
\emph{Probability inequalities for the sum in sampling without replacement}
The Annals of Statistics 2 (1974):  39--48.

\bibitem{sperner}
E. Sperner.
\emph{Ein Satz \"uber Untermengen einer endlichen Menge.}
Math. Zeit., 27 (1928): 544--548.

\bibitem{tsai}
S.-F. Tsai.
\emph{A simple upper bound on the number of antichains in $[t]^n$.}
Order 36 (2019): 507--510.

\bibitem{tomon15}
I. Tomon. 
\emph{On a conjecture of F\"uredi.}
European Journal of Combinatorics 49 (2015): 1--12.

\bibitem{T19}
I. Tomon. 
\emph{Forbidden induced subposets of given height.}
Journal of Combinatorial Theory, Series A 161 (2019): 537--562.

\bibitem{T20}
I. Tomon. 
\emph{Packing the Boolean lattice with copies of a poset.}
Journal of London Mathematical Society 101 (2) (2020): 589--611.

\end{thebibliography}
\end{document}